\documentclass[a4paper,11pt]{article}
\usepackage{amsmath,amssymb}
\usepackage{amsfonts,graphicx,latexsym}
\usepackage{tikz}
\usetikzlibrary{calc,positioning,scopes}
\usetikzlibrary{arrows}
\usepackage{float}
 
\usepackage{changes}

\def\qed{$\Box$}                              

\newenvironment{proof}{\noindent{\sc
   Proof:}\enspace}{\hfill\hfill\qed\medskip}
 
\newcommand{\R}     {\mathbb{R}} 

\usepackage{mathtools, xparse}

\def\1{{\mathchoice {1\mskip-4mu\mathrm l}      
{1\mskip-4mu\mathrm l} 
{1\mskip-4.5mu\mathrm l} {1\mskip-5mu\mathrm l}}} 
 
\def\comment#1{} 
 
\newtheorem{theorem}{Theorem}[section] 
\newtheorem{Lemma}[theorem]{Lemma} 
\newtheorem{proposition}[theorem] {Proposition} 
\newtheorem{cor}[theorem]  {Corollary} 
\newtheorem{remark}[theorem]  {Remark}

\newcommand{\ERRW}{edge-reinforced random walk}
\newcommand{\VRRW}{vertex-reinforced random walk}
\newcommand{\FF}{\mbox{${\cal F}$}}
 
\newcommand{\GG}{\mbox{${\cal G}$}}
\newcommand{\nn}{{\bar{n}}}
\newcommand{\NN}{\mbox{${\cal N}$}}
\newcommand{\NNs}{{\cal N}}
\newcommand{\bN}{\mathbb{N}}

\DeclarePairedDelimiter{\norm}{\lVert}{\rVert}

\begin{document}

\title{Edge- and vertex-reinforced random walks with super-linear reinforcement on infinite graphs
}

\date{}
\author{
Codina Cotar
\footnote{
University College London, Statistical Science Department,
London,
United Kingdom,
\texttt{c.cotar@ucl.ac.uk}}
\, and
Debleena Thacker
\footnote{
Mathematical Statistics, Centre for Mathematical Sciences, Lund University, Lund, Sweden
\texttt{thackerdebleena@gmail.com}}}

\maketitle

\begin{abstract}
In this paper we introduce  a new simple but powerful general technique for the study of edge- and vertex-reinforced processes with super-linear reinforcement, based on the use of order statistics for the number of edge, respectively of vertex, traversals. The technique relies on upper bound estimates for the number of edge traversals, proved in a different context by Cotar and Limic [\textit{Ann. Appl. Probab.} (2009)] for finite graphs with edge reinforcement.  We apply our new method both to edge- and to vertex-reinforced random walks with super-linear reinforcement on arbitrary infinite connected graphs of bounded degree. We stress that, unlike all previous results for processes with super-linear reinforcement, we make no other assumption on the graphs.

For edge-reinforced random walks, we complete the results of Limic and Tarr\`{e}s [\textit{Ann. Probab.} (2007)] and we settle a conjecture of Sellke [\textit{Technical Report 94-26, Purdue University} (1994)] by showing that for any reciprocally summable reinforcement weight function $w$, the walk traverses a random attracting edge at all large times.

For vertex-reinforced random walks, we extend results previously obtained on $\mathbb{Z}$ by Volkov [\textit{Ann. Probab.} (2001)] and by Basdevant, Schapira and Singh [\textit{Ann. Probab.} (2014)], and on complete graphs by Benaim, Raimond and Schapira [\textit{ALEA} (2013)]. We show that on any infinite connected graph of bounded degree, with reinforcement weight function $w$
taken from a general class of reciprocally summable reinforcement
weight functions, the walk traverses two random neighbouring attracting vertices at all large
times.
\end{abstract}

\smallskip
\noindent {\bf AMS 2000 subject classification:} 60G50, 60J10, 60K35 
\bigskip

\textit{Key words and phrases:}         
edge-reinforced random walk, vertex-reinforced random walk, super-linear (strong) reinforcement, attraction set, order statistics, Rubin's construction, bipartite graphs

\section{Introduction}
Let $\GG$ be a locally finite connected graph with the 
edge set $E(\GG)$ and the vertex set $V(\GG)$. We call any two vertices $u,v$ connected by an edge \textit{adjacent}
(or \textit{neighbouring}) vertices; in this case we
write $u\sim v$ and denote by $\{u,v\}=\{v,u\}$ the
edge connecting them.
We will denote by  $|E(\GG)|$
the number of edges of $\GG$, and by $|V(\GG)|$
the number of vertices of $\GG$.
We will denote by $D(\GG)=\sup_{v\in V({\scriptsize \GG})}\mbox{degree}(v)<\infty$ the
degree of $\GG$, where for any $v\in V(\GG)$, the degree$(v)$ equals the 
number of edges incident to $v$.

Finally, we will denote by ${\mathbb{P}}^{\GG}$ 
the law of the edge/vertex-reinforced random walk on $\GG$, and by ${\mathbb{E}}^{\GG}$ the corresponding expectation. 

We will next introduce the edge/vertex reinforced random walks.

\subsection{Edge-reinforced random walk}

Let $(\ell_0^e,\, e\in E(\GG))$ be given arbitrary real numbers with
$\ell_0^e\geq 0 $, $ e\in E(\GG)$; these are the initial edge weights. 
Given a \textit{reinforcement weight function} 
$w: [0,\infty) \mapsto (0,\infty)$,
the edge-reinforced random walk (ERRW)
on $\GG$ records nearest-neighbour step transitions of a particle in $V(\GG)$. That is:
\begin{itemize}
\item [(i)] if currently at vertex $v\in V(\GG)$, in the next step the particle jumps to a vertex $u\in V(\GG)$ adjacent to $v$.
\item [(ii)] the probability of a jump to $u$ is $w$-\textit{proportional} to the number of previous traversals of the edge $\{v,u\}$. 
\end{itemize}
The more formal definition is as follows. Denote by
$I_n$ the (\textit{random}) position of the {\ERRW} at time $n$. At the \textit{initial time} $t_0:=0$, the {\ERRW} is at the initial position $I_{t_0}\in V(\GG)$, and 
$\{I_n, I_{n+1}\} \in E({\GG})$ for all $n\geq t_0$.
Let $\FF_n$ be the filtration
\begin{equation}
\label{Efilt}
\FF_n=\sigma\{I_k,k=0,\ldots ,n, (\ell_0^e,e\in E({\GG}))\}.
\end{equation}
Moreover, the {\ERRW} follows the following rule for all $n\ge t_0$: 
\begin{equation}
\label{ruleerrw}
  {\mathbb{P}}^{\GG}(I_{n+1}=v|\FF_n)  1_{\{I_n=u \}}= \frac{w(X_n^{\{u,v\}})}{
  \sum_{y\sim u}w(X_n^{\{u,y\}})} 1_{\{I_n=u, u\sim v \}},
\end{equation}
where for any $e\in E(\GG)$ and for all $n\ge t_0+1$
\begin{equation}
\label{Esteps}
X_n^e = \ell_0^e + \sum_{i=t_0+1}^{n} 
1_{\{ \mbox{\small{$e$ was traversed at $i$th step}} \}} =
\ell_0^e + \sum_{i=t_0+1}^{n}
1_{\{ \{I_{i-1}, I_{i}\}=e \}} 
\end{equation}
equals the \textit{initial edge weight} $\ell_0^e$ incremented
by the total number of 
(undirected) traversals of edge $e$ up to time $n$.
Note that, whenever $|V(\GG)|< \infty$,
$\sum_{e\in E({\scriptsize \GG})}X_k^e = k+\sum_{e\in E(\GG)}\ell_0^e$ for all $k\geq 0$, almost surely.
The starting weights $X_{t_0}^e:=\ell_0^e$ are specified as 
deterministic above but one could use random variables instead
in applications, and the definition (\ref{Efilt}) accounts for this possibility.

\subsection{Vertex-reinforced random walk}

Let $(\ell_0^v,\, v\in V(\GG))$ be given arbitrary real numbers with
$\ell_0^v\geq 0 $, $ v\in V(\GG)$; these are the initial vertex weights. 
Given a \textit{reinforcement weight function} 
$w:[0,\infty) \mapsto (0,\infty)$,
the vertex-reinforced random walk (VRRW)
on $\GG$ records nearest-neighbour step transitions of a particle in $V(\GG)$. That is:
\begin{itemize}
\item [(i)] if currently at vertex $v\in V(\GG)$, in the next step the particle jumps to a vertex $u\in V(\GG)$ adjacent to $v$.
\item [(ii)] the probability of a jump to $u$ is $w$-\textit{proportional} to the number of previous traversals of the vertex $u$. 
\end{itemize}
The more formal definition is as follows.
Denote by
$I_n$ the (\textit{random}) position of the {\VRRW} at time $n$. At the \textit{initial time} $t_0:=0$, the {\VRRW} is at the initial position $I_{t_0}\in V(\GG)$,
and 
$\{I_n, I_{n+1}\} \in E({\GG})$ for all $n\geq t_0$.
Let $\FF_n$ be the filtration
\begin{equation}
\label{Efiltv}
\FF_n=\sigma\{I_k,k=0,\ldots ,n, (\ell_0^v,v\in V({\GG}))\}.
\end{equation}
Moreover, the {\VRRW} follows the following rule for $n\ge t_0$: 
\[
  {\mathbb{P}}^{\GG}(I_{n+1}=v|\FF_n)  1_{\{I_n=u \}}= \frac{w(X_n^{v})}{
  \sum_{y\sim u}w(X_n^{y})} 1_{\{I_n=u, u\sim v \}},
\]
where for any $v\in V(\GG)$ and $n\ge t_0+1$ 
\begin{equation}
\label{Esteps}
X_n^v = \ell_0^v + \sum_{i=t_0+1}^{n} 
1_{\{ \mbox{\small{the walk jumped to $v$ at the $i$th step}} \}} =
\ell_0^v + \sum_{i=t_0+1}^{n}
1_{\{I_{i}=v \}} 
\end{equation}
equals the \textit{initial vertex weight} $\ell_0^v$ incremented
by the total number of jumps to the vertex $v$ up to time $n$.
Whenever $|V(\GG)|< \infty$,
$\sum_{v\in V({\scriptsize \GG})}X_k^v = k+\sum_{v\in V(\GG)}\ell_0^v$ for all $k\geq 0$, almost surely.

\subsection{Main results}

We will state below the main results for each of our models of interest introduced above.

\subsubsection{Edge-reinforced random walk}

We assume for arbitrary initial edge weights $\ell_0^e\ge 0,e\in E(\GG)$, the condition on $w$
\begin{equation}
 \label{Azero}
\sup_{e\in E(\GG)}\sum_{i=1}^\infty\frac{1}{w(i+\ell_0^e)}<\infty.
\end{equation}
Any weight $w$ satisfying condition (\ref{Azero}) is called \textit{super-linear (or strong)}, and the corresponding ERRW is called \textit{strongly reinforced walk}.

Let $\mathbb{N}$ be the set of non-negative integers. If $\ell_0^e\in\mathbb{N}, e\in E(\GG),$ it is sufficient to assume instead of (\ref{Azero}) the assumption that
\begin{equation}
\label{Azeronat}
 \sum_{i =1 }^\infty \frac{1}{w(i)}<\infty,
\end{equation}
since for all $\ell_0^e\in\mathbb{N}$, we have
$$\sum_{\ell=1}^\infty \frac{1}{w(\ell+\ell_0^e )} \leq \sum_{i =1}^\infty \frac{1}{w(i)}<\infty.$$
We also make the additional assumption that 
\begin{equation}
\label{initialweight}
\sup_{e\in E(\GG)} w(\ell_0^e)<\infty.
\end{equation}
Recall that  $X_k^e$ equals the initial edge weight
$\ell_0^e$ incremented by the number of times edge $e$ has been visited by time $k$. Let
\begin{equation}
\label{attr}
\GG_\infty=\{e\in E:\sup_n X_n^e=\infty\}
\end{equation}
be the (random) graph spanned by all edges in $\GG$ that are traversed by the
walk infinitely often. As we will show below, we have
\begin{equation*}
\{\GG_\infty~\mbox{has only one edge}\}=\{\exists N<\infty~\mbox{such that}~\{I_n,I_{n+1}\}=\{I_n,I_{n-1}\}~\forall n\ge N\}.
\end{equation*}
Our main result for edge-reinforced random walks is 
\begin{theorem}
\label{maininfedge}
Let $\GG$ be an infinite connected graph of bounded degree. If $w$ satisfies (\ref{Azero}) and (\ref{initialweight}), the edge-reinforced
random walk on $\GG$ traverses a random attracting edge at all large times a.s.,
that is
\begin{equation}
\label{edgetrav}
{\mathbb{P}}^{\GG}(\GG_\infty~\mbox{has only one edge}) = 1.
\end{equation}
\end{theorem}
We emphasize that, unlike all other existing results in the literature on super-linear ERRW, we make \textit{no other assumptions} on the graphs except that they are connected and of bounded degree. Theorem \ref{maininfedge} proves a long-standing conjecture of Sellke \cite{sel94} regarding the generality of the graphs and of the reinforcement weight $w$ under which (\ref{edgetrav}) holds. Moreover, our proof can be extended to the case where each edge follows its own reinforcement function $w_e$ satisfying (\ref{Azero}) and (\ref{initialweight}), as explained in detail in Remarks \ref{difw} (a) and \ref{eachedgeinf} (a) below.

We next briefly discuss the link of our work to the recent literature.
For a detailed review of a number of interesting results on {\ERRW}, we refer the reader to a survey of Pemantle \cite{pem_survey06} on stochastic reinforcement processes.

A result of Sellke \cite{sel94} 
implies that (\ref{Azeronat}) is sufficient and necessary for 
\begin{equation}
\label{Eatrac}
{\mathbb{P}}^{\GG}(\mbox{the walk ultimately traverses a single edge})=1,
\end{equation}
whenever the underlying graph is bipartite and of bounded degree.
More recently, Limic and Tarr\`es \cite{LT05} showed that for a large class of weight functions $w$, which includes the class of increasing functions satisfying (\ref{Azeronat}), 
(\ref{Eatrac}) holds on any graph of bounded degree. Two examples of almost increasing weights $w$ for which the method in \cite{LT05} does not work on finite/infinite graphs with {at least} an odd cycle are: $w(k)=k^{1+\rho}/(2+(-1)^k), \rho>0,$ and $w(k)=\exp\{k(2+(-1)^k\}$.

We consider next the optimality of our assumptions. It is easy to find examples of locally bounded trees with $D(\GG)=\infty$
such that (\ref{Azero}) holds but that the range of the walk is infinite with positive probability.
Morever, Sellke \cite{sel94} provides (slightly peculiar) examples  of {\ERRW s} on $\mathbb{Z}$
where $\sum_k 1/w(k)$ is finite over even $k$ and infinite over odd $k$, but where 
$\GG_\infty$ is still a finite graph, almost surely. On the other hand, for the related generalized Polya urn model with infinitely many urns, Collevecchio, Cotar and LiCalzi \cite{CCL} provide an example of super-linear weights $w_e,e\in E(\GG),$ which satisfy instead of (\ref{Azero}) only the weaker assumption $\sum_{i=1}^\infty\frac{1}{w_{e}(i+\ell_0)}<\infty,e\in E(\GG), \ell_0\in\mathbb{N}$, and do not satisfy (\ref{initialweight}), and in which case the walk visits all urns finitely many times.

\subsubsection{Vertex-reinforced random walk}

We assume for arbitrary initial vertex weights $\ell_0^v\ge 0,v\in V(\GG)$, the condition on $w$
\begin{equation}
 \label{Azerov}
\sup_{v\in V(\GG)}\sum_{i=1}^\infty\frac{1}{w(i+\ell_0^v)}<\infty.
\end{equation}
As in the edge-reinforced case, if $\ell_0^v\in\mathbb{N}, v\in V(\GG),$ it is sufficient to assume instead of (\ref{Azerov}) that
\begin{equation}
\label{Azerovnat}
 \sum_{i =1 }^\infty \frac{1}{w(i)}<\infty.
\end{equation}
We also make the additional assumption
\begin{equation}
\label{initialweightv}
\sup_{v\in V(\GG)} w(\ell_0^v)<\infty.
\end{equation}
Recall that  $X_k^v$ equals the initial vertex weight
$\ell_0^v$ incremented by the number of times vertex $v$ has been visited by time $k$. Let
\begin{equation}
\label{attr}
\GG_\infty=\{v\in E:\sup_n X_n^v=\infty\}
\end{equation}
be the (random) graph spanned by all vertices in G that are traversed by the
walk infinitely often. 

Our main result for vertex-reinforced random walks is 
\begin{theorem}
\label{maininfvert}
Let $\GG$ be an infinite connected graph of bounded degree. If $w$ satisfies (\ref{initialweightv}) and
\begin{equation}
\label{supinfv}
\sup_{v\in V(\GG)}\sum _{i\geq 1}^{\infty}\frac{i}{w(i+\ell_0^v)} < \infty,
\end{equation}
the walk traverses exactly two random neighbouring attracting vertices at all large times a.s.,
that is
$${\mathbb{P}}^{\GG}(\GG_\infty~\mbox{has exactly two vertices}) = 1.$$
\end{theorem}
Just as for edge-reinforced random walks, for our theorem we make \textit{no other assumptions} on the graphs except that they are connected and of bounded degree. Moreover, we can again generalize the proof to the case where each vertex follows its own reinforcement function $w_v$ satisfying (\ref{initialweightv}) and (\ref{supinfv}). To the best of our knowledge, our theorem is the first result of almost sure attraction to two vertices for vertex-reinforced walks with super-linear reinforcement on general finite/infinite graphs of bounded degree. It has been previously shown by Benaim, Raimond and Schapira \cite{BRS} that VRRW with super-linear reinforcement on infinite graphs of bounded degree gets stuck almost surely on a finite graph.  

VRRW with super-linear reinforcement turns out to have a more complicated structure than ERRW and there are few results available. It has been studied on $\mathbb{Z}$ for non-decreasing weights $w$ satisfying (\ref{Azerov}) by Volkov \cite{Vol2}  and by Basdevant, Schapira and Singh \cite {BSS}; therein, the authors showed that the walk gets stuck on two vertices. This is in contrast to the recent result of Benaim, Raimond and Schapira \cite{BRS} for complete graphs with weight function $w(\ell)=\ell^\rho, \rho>1$, in which particular case the walk is shown to get stuck with positive probability on more than 2 vertices if $1<\rho\le 2$. 
\begin{remark}
\label{conjecture}
\begin{itemize} 
\item [(a)] The assumptions in Theorem \ref{maininfvert} are satisfied by a large class of weight functions, among which are all the weight functions of order equal to or higher than $w(\ell)=\ell^2\log^2\ell$. Indeed, we have in this case for $\ell_0^v\in\mathbb{N},v\in V(\GG)$, that
$$\sup_{v\in V(\GG)}\sum _{i\geq 1}^{\infty}\frac{i}{w(i+\ell_0^v)} \le\sum _{i=1}^{\infty}i/w(i)\le O\big(\sum_{i=1}^\infty 1/ i\log^2 i\big)<\infty.$$

\item [(b)] We do not expect Theorem \ref{maininfvert} to hold for vertex-reinforced walks on infinite graphs of bounded degree in full generality of graph and of super-linear reinforcement functions $w$. In fact, we believe that (\ref{supinfv}) is almost optimal for general graphs for the statement in Theorem \ref{maininfvert}. Instead, for $w $ satisfying (\ref{Azerov}) but not (\ref{supinfv}), we conjecture that the size of the attraction set will depend on the geometry of the graph, in particular on whether the graph has any triangles, and on the weight $w$, as is the case in \cite{BRS} and \cite{Vol2}. 

\item [(c)] We also conjecture, based on preliminary estimates, that the following holds on arbitrary infinite graphs of bounded degree: For any fixed $k\ge 2$, if
$$\sup_{v\in V(\GG)}\sum _{i\geq 1}^{\infty}\frac{i^{1/k}}{w(i+\ell_0^v)} < \infty,$$
then the walk traverses at most $k+1$ random attracting vertices at all large times a.s.,
that is
$${\mathbb{P}}^{\GG}(\GG_\infty~\mbox{has at most}~k+1~\mbox{vertices}) = 1.$$
On complete graphs with $|V(\GG)|=\nn$, this has been shown in \cite{BRS} for $k\le\nn-1$ and $w(\ell)=\ell^\rho, 1<\rho\le 2$.
\item [(d)] By means of Proposition \ref{Pfingraph1v} below, we can also compute for super-linear VRRW in finite graphs satisfying (\ref{supinfv}) the asymptotic behaviour of the tail distribution of the (random) time of attraction to two vertices. Formally, the {time of attraction} $T$ is defined as follows
\begin{eqnarray}
\label{ETsym}
T
&=&\inf\{k\geq 0: \{I_n,I_{n+1}\}=\{I_{n+1},I_{n+2}\}, \forall n\geq k\},
\end{eqnarray}
that is, the first time after which only the {attracting}
 edge is traversed.

The arguments are similar to the ones for the super-linear ERRW case studied in \cite{CotarLimic}.
\end{itemize}
\end{remark}
In the special case of bipartite graphs, in particular of ${\mathbb{Z}}^d, d\ge 1$, our method yields a stronger result than in Theorem \ref{maininfvert} above, which includes also the results obtained for $\mathbb{Z}$ both in \cite{Vol2}  and in \cite {BSS}. 
\begin{cor}
\label{zd}
Let $\GG$  be an infinite bipartite graph of bounded degree. If $w$ satisfies (\ref{initialweightv}) and either: 
\begin{itemize}
\item [(a)] \begin{equation}
\label{supinfv1}
\sup_{v\in V(\GG)}\sum _{i\geq 1}^{\infty}\frac{i^{1/2}}{w(i+\ell_0^v)} < \infty,~~\mbox{or}
\end{equation}
\item [(b)] $w$ satisfies (\ref{Azerov}) and 
$$\sup_{v\in V(\GG)} \sup_{i\ge 1}\frac{i}{w(i+\ell_0^v)}<\infty,$$
\end{itemize}
then the walk traverses exactly two random neighbouring attracting vertices at all large times a.s..
\end{cor}
\begin{remark}
The assumptions in (a) above are satisfied by a large class of weight functions, among which are all the weight functions of order equal to or higher than $w(\ell)=\ell^{3/2}\log^2\ell$. The assumptions in (b) are satisfied in particular by all non-decreasing weights $w$ which obey (\ref{Azerov}).
\end{remark}
We can actually extend the result in Corollary (\ref{zd}) (a) to the more general case of triangle-free graphs. More precisely, we can show
\begin{cor}
\label{tfg}
Let $\GG$ be an infinite triangle-free graph, connected and of bounded degree. If $w$ satisfies (\ref{initialweightv}) and
\begin{equation}
\label{supinfv1}
\sup_{v\in V(\GG)}\sum _{i\geq 1}^{\infty}\frac{i^{1/2}}{w(i+\ell_0^v)} < \infty,
\end{equation}
then the walk traverses exactly two random neighbouring attracting vertices at all large times a.s..
\end{cor}

\subsubsection{Strategy of the proof}
We develop for the proof of our main statements  a new simple but powerful general technique for the study of edge- and vertex-reinforced processes, based on the use of the order statistics for the number of edge traversals (for ERRW), respectively of vertex traversals (for VRRW), in the graph. 

The key observation for our proofs on arbitrary finite graphs is that showing that the walk gets attracted to \textit{at most $i$ vertices} is equivalent to showing that the expectation $\mathbb{E}\left(g(R^{i+1})\right)$ is finite, where $g$ is a given measurable function and $R^{i+1}$ is the $(i+1)$th largest number of edge, respectively of vertex, traversals. By monotonicity arguments, this reduces to the question of finding a tight upper bound for the distribution of $R_k^{i+1}$, i.e. the $(i+1)$th largest number of edge, respectively of vertex, traversals by time $k$. The main tool for obtaining this bound is an upper bound inequality for the number of edge traversals at time $k$, proved in a different context by Cotar and Limic in \cite{CotarLimic} for all finite graphs with edge reinforcement.  

To prove our results on general infinite graphs of bounded degree, we couple our statements on finite graphs with an inequality giving upper bounds on the probability that the walk ever visits more than $\nn$ edges, respectively vertices. Of crucial importance now is that our finite graph theorems hold on any finite connected graphs, with no restriction on their particular geometric properties.

\vspace{2mm}

The rest of the paper is organized as follows. In Section 2 we focus our attention on the proof of Theorem \ref{maininfedge}, our main edge-reinforced random walks result; in Subsection 2.1 we prove in Theorem \ref{mainedgefin} the result on arbitrary finite graphs and in Subsection 2.2 we extend our arguments by means of Lemma \ref{Lem:Stuck_n_distance}  to ERRW on infinite graphs of bounded degree.  In Section 3 we show Theorem \ref{maininfvert}, our main vertex-reinforced random walk result; in Subsection 3.1 we prove the result on arbitrary finite graphs in Theorem \ref{mainvertfin} and in Subsection 3.2 we extend the reasoning by using Lemma \ref{Lem:Stuck_n_distance_vertex_a} to VRRW on infinite graphs of bounded degree.

\section{Strongly edge-reinforced random walks on general graphs}
\label{SGenGraph}

In this section we will prove Theorem \ref{maininfedge}. We will first work out the strategy of the proof for finite graphs in Subsection \ref{SFGraph} and then we will use the finite graph computations to extend our arguments in Subsection \ref{InfGraphs} to the case of infinite connected graphs with bounded degree.

\subsection{Analysis on finite graphs}
\label{SFGraph}
Let $\GG$ be a finite graph, and abbreviate $\nn=|E(\GG)|$. Denote the set of edges of $\GG$ by $E(\GG)=\{e_1,e_2,\ldots e_\nn\}.$ 
If $v$ is an arbitrary vertex of the graph, let $d_v:=\mbox{degree(v)}$, and let
$\NN_v:=\{e_1^v, e_2^v,\ldots e_{d_v}^v\}$ be the set of  edges incident to $v$.

\subsubsection{Bounds for the probabilities of the edge weights order statistics}

Recall that $t_0:=0$.
Fix the initial position $I_{t_0}$ at some arbitrary vertex $v_0$. 
We re-label $(X_k^e-\ell_0^e)$ (the number of edge traversals at time $k\ge 0$) in increasing order. More precisely, we define the \textit{order statistics at time $k$} as a vector $R_k = (R_k^1,...,R_k^{\nn})$.  The components of this vector are defined as the values of 
$$e\mapsto  X_k^e - \ell_0^e$$
put in non-increasing order; this defines the vector $R_k$ uniquely. Therefore, for all $k\ge 0$ we have by definition
$$0\le R_k^\nn\le R_k^{\nn-1}\le\ldots\le R_k^{1}\le k,~~~R_k^i\le R_{k+1}^i~\mbox{for all}~i=1,2,\ldots,\nn,~~\mbox{and}~~\sum_{j=1}^{\nn}R_k^j=k.$$ 

We recall next Proposition 19, proved by Cotar and Limic in \cite{CotarLimic}, which \textit{crucially} gives \textit{path-independent} upper bounds on the distribution of the number of edge traversals at time $k$.
\begin{proposition}
\label{Pfingraph1}
Let $k\geq 1$ and
$v\in V(\GG)$ and denote by $A_{v,k}$ the event $\{I_k=v\}$.
Then for any $\ell_k^e\in\mathbb{N}$, $e\in E(\GG)$, such that 
$\sum_{e \in E({\scriptsize \GG})}\ell_k^e=k$, we have 
\begin{equation}
\label{EPfingraph1}
{\mathbb{P}}^{\GG}(X_k^e-\ell_0^e=\ell_k^e, e \in E(\GG), A_{v,k})
\leq 
\frac{\prod_{e\in E({\scriptsize \GG})} w(\ell_0^e)}
{\min_{e \in {\NN}_{v_0}}\, w(\ell_0^e) }
\cdot 
\frac{\sum_{e\in \NNs_v}w(\ell_k^e+\ell_0^e)}
{\prod_{e\in E({\scriptsize \GG})} w(\ell_k^e+\ell_0^e)}.
\end{equation}
\end{proposition}
Before we state a similar proposition to the above for $R_k$, we will first introduce some more notation. We re-label the initial edge weights $(\ell_0^e)_{e\in E(\GG)}$ by  $(\ell_0^{e_i})_{ e_i\in E(\GG), i=1,\ldots,\nn}$. Denote by $S(\nn)$ the symmetric group, i.e. the set of all permutations of $\{1,2,\ldots,\nn\}$, and by $\sigma$ an arbitrary element in $S(\nn)$. From Proposition \ref{Pfingraph1} we have that
\begin{proposition}
\label{edgegen}
Let $k\geq 1$ and
$v\in V(\GG)$ and denote by $A_{v,k}$ 
the event $\{I_k=v\}$.
Then for any $\ell_k^i\in\mathbb{N}, i=1,\ldots,\nn,$ such that 
$0\le\ell_k^{\nn}\le\ldots \le\ell_k^i\le\ldots \ell_k^1\le k$ and
$\sum_{i=1}^{\nn}\ell_k^i=k$, we have 
\begin{equation}
\label{EPfingraph2}
{\mathbb{P}}^{\GG}(R_k^i=\ell_k^i, i=1,\ldots, \nn, A_{v,k})
\leq \frac{\prod_{e\in E(\GG)} w(\ell_0^e)}
{\min_{e \in E(\GG)}\, w(\ell_0^e)}
\cdot 
\sum_{\sigma\in S(\nn)}\frac{\sum_{i=1}^{\nn}w\left(\ell_k^{i}+\ell_0^{e_{\sigma(i)}}\right)}
{\prod_{i=1}^{\nn} w\left(\ell_k^i+\ell_0^{e_{\sigma(i)}}\right)}.
\end{equation}
and 
\begin{equation}
\label{EPfingraph3}
{\mathbb{P}}^{\GG}(R_k^i=\ell_k^i, i=1,\ldots, \nn)
\leq |V(\GG)| \cdot\frac{\prod_{e\in E(\GG)} w(\ell_0^e)}
{\min_{e \in E(\GG)}\, w(\ell_0^e)}
\cdot 
\sum_{\sigma\in S(\nn)}\frac{\sum_{i=1}^{\nn}w\left(\ell_k^{i}+\ell_0^{e_{\sigma(i)}}\right)}
{\prod_{i=1}^{\nn} w\left(\ell_k^i+\ell_0^{e_{\sigma(i)}}\right)}.
\end{equation}
\end{proposition}
\begin{proof}
We will only prove (\ref{EPfingraph2}) as (\ref{EPfingraph3}) follows immediately from (\ref{EPfingraph2}) by summing over all possible vertices $A_{v,k}, v\in V(\GG)$. 

We note first that the event $\{R_k^i=\ell_k^i, i=1,\ldots, \nn\}$ is a union over $\sigma\in S(\nn)$ of the events $\{X_k^{e_{i}}-\ell_0^{e_{i}}=\ell_k^{\sigma(i)}, i=1,\ldots,\nn\}$.
Assume $\NN_v=\{e_{j_1},\ldots, e_{j_{d_v}}\}, j_1,\ldots, j_{d_v}\in \{1,\ldots, \nn\}$,  with $j_1\neq\ldots\neq j_{d_v}$. Then
\begin{eqnarray*}
{\mathbb{P}}^{\GG}(R_k^i=\ell_k^i, i=1,\ldots, \nn, A_{v,k})&\le&\sum_{\sigma\in S(\nn)} {\mathbb{P}}^{\GG}(X_k^{e_i}-\ell_0^{e_i}=\ell_k^{\sigma(i)}, i=1,\ldots,\nn, A_{v,k})\\
&\le&\sum_{\sigma\in S(\nn)}\,\,\frac{\prod_{i=1}^{\nn} w(\ell_0^i)}
{\min_{e \in {\NN}_{v_0}}\, w(\ell_0^e)}
\cdot 
\frac{\sum_{s=1}^{d_v}w(\ell_k^{\sigma(j_s)}+\ell_0^{e_{j_s}})}
{\prod_{i=1}^{\nn} w\left(\ell_k^{\sigma(i)}+\ell_0^{e_i}\right)}\\
&\le&\frac{\prod_{e\in E(\GG)} w(\ell_0^e)}
{\min_{e \in E(\GG)}\, w(\ell_0^e)}\cdot\sum_{\sigma\in S(\nn)}\frac{\sum_{i=1}^{\nn}w\left(\ell_k^{\sigma(i)}+\ell_0^{e_i}\right)}
{\prod_{i=1}^{\nn} w\left(\ell_k^{\sigma(i)}+\ell_0^{e_i}\right)},
\end{eqnarray*}
where for the second inequality in the above we used (\ref{EPfingraph1}), for the third inequality we used that
$$\prod_{e\in E(\GG)} w(\ell_0^e)=\prod_{i=1}^{\nn} w(\ell_0^{e_i})~~~~\mbox{and}~~~~\sum_{s=1}^{d_v}w(\ell_k^{\sigma(j_s)}+\ell_0^{e_{j_s}})\le \sum_{i=1}^{\nn}w\left(\ell_k^{\sigma(i)}+\ell_0^{e_i}\right),$$ 
and that $\min_{e \in E(\GG)}\, w(\ell_0^e)\le \min_{e \in {\NN}_{v_0}}\, w(\ell_0^e)$. The statement in  (\ref{EPfingraph2}) follows now by noting that
$$\sum_{\sigma\in S(\nn)}\frac{\sum_{i=1}^{\nn}w\left(\ell_k^{\sigma(i)}+\ell_0^{e_i}\right)}
{\prod_{i=1}^{\nn} w\left(\ell_k^{\sigma(i)}+\ell_0^{e_i}\right)}=\sum_{\sigma\in S(\nn)}\frac{\sum_{i=1}^{\nn}w\left(\ell_k^{i}+\ell_0^{e_{\sigma(i)}}\right)}
{\prod_{i=1}^{\nn} w\left(\ell_k^{i}+\ell_0^{e_{\sigma(i)}}\right)}.$$
\end{proof}
\begin{remark}
\label{reducew}
\begin{itemize}
\item [(a)] Looking carefully at (\ref{EPfingraph2}) and (\ref{EPfingraph3}), we notice that the only non-initial weights which contribute to the bounds for the order statistics probabilities are those for edges which are traversed at least once up to time $k$. This is because the edges that are not traversed
have $\ell_k^i =0$, and therefore the associated weights cancel between the product in the numerator and the product in the denominator. However, in our computations below, we will need to take into account all the weights in the bound, including those for non-traversed edges, as we will sum over all possible values for the order statistics of the edge weights.

\item [(b)] Since in computing the various probabilities in Proposition \ref{Pfingraph1} by means of Proposition 19 from \cite{CotarLimic} we can re-scale $w(k)$ to $w(k)/w_{\max}(\ell_0),$ where $w_{\max}(\ell_0):=\max_{e\in E(\GG)} w(\ell_0^e)$, we will take in our computations below $w(\ell_0^e)\le 1$ for all $e\in E(\GG)$. This means that the term $\prod_{e\in E(\GG)} w(\ell_0^e)$ from the formulas in Proposition \ref{edgegen} will not appear in our computations below.

\item [(c)] For star graphs, (\ref{EPfingraph1}) is optimal for super-linear weights, in the sense that there is also a lower bound of the same order, as proved in Proposition 2 from \cite{CotarLimic}.
\end{itemize}
\end{remark}

\subsubsection{Attraction to one edge on finite graphs}

We will first give an alternative definition of the event $\{\GG_\infty~\mbox{has only one edge}\}$ by using the order statistics for the number of edge traversals. First, since by definition the sequence of random variables $R_k^{2}, k\ge 1,$ is non-decreasing  in $k$ there exists $R_\infty^{2}=\lim_{k\rightarrow\infty} R_k^{2}$, which may or may not be finite. We have immediately for any finite graph that
$$\{\GG_\infty ~\mbox{has only one edge}\}=\{R_\infty^2<\infty\}.$$
We will use our alternative definition above to prove in this section
\begin{theorem}
\label{mainedgefin}
Let $\GG$ be a finite graph with $\nn\ge 3$ edges. If $w$ satisfies (\ref{Azero}), the edge-reinforced
random walk on $\GG$ traverses a random attracting edge at all large times a.s.,
that is
$${\mathbb{P}}^{\GG}(\GG_\infty~\mbox{has only one edge}) = 1.$$
\end{theorem}
In preparation for the proof of attraction to one edge, we will first obtain upper bounds for the distribution of $R_k^{2},$ that is, the distribution of the 2nd largest number of edge traversals at time $k$. For simplicity and clarity of computations, we will consider in detail only the case where all edges have \textit{equal} initial weights $\ell_0^e\equiv\ell_0\in\R^+$, the case with general initial weights $\ell_0^e\in\R^+,e\in E(\GG),$ following by similar arguments by means of Proposition \ref{edgegen}.  

Throughout the paper, we will denote by $[x]$ the integer part of $x$. Moreover, for all $a,c, m\in\mathbb{N},$ $m\ge 1$, and for all $b\in\R^+$, we will denote by
\begin{equation}
\label{eqndefq}
Q_{m}(a; b;c):=\sum_{(h^1,\ldots, h^m)\in\mathbb{N}^m\,:\,0\le h^m\le\ldots\le h^1\le c\atop h^1+\ldots+h^m=a} \frac{1}{\prod_{j=1}^{m} w(b+h^j)},
\end{equation}
with the convention that $Q_{m}(a; b;c)=0$ if there exists no $(h^1,\ldots, h^m)\in \mathbb{N}^m$\ with $0\le h^m\le\ldots\le h^1\le c$ and $h^1+\ldots+h^m=a$.
\begin{proposition}
\label{2orderstat}
Assume that $w$ satisfies (\ref{Azero}). We will show
\begin{itemize}
\item [(a)] For all $m, a, c, j\in\mathbb{N}, m\ge 1$, and for all $b\in\R^+$, we have 
\begin{equation}
\label{rec2}
\sum_{s=0}^{j}Q_{m}(s+a; b;c)\le \sum_{s=0}^{j}Q_{m}(s+a; b;\infty)\le (c(b))^m,~~\mbox{with}~~c(b):= \sum_{\ell=0}^\infty \frac{1}{w(\ell+b)}.
\end{equation}
\end{itemize}
\item [(b)] Let $\nn\ge 3$. For $\ell_0\in\R^+$, let $\ell^e_0 = \ell_0$ for all edges $e\in E(\GG)$. Then for all $k\ge 1$ and all $0\le\ell_k^{2}\le [k/2], \ell_k^2\in\mathbb{N}$, we have for some $C(w,\nn, |V(\GG)|,\ell_0)>0$ which depends only on $w,\nn, |V(\GG)|$ and $\ell_0$ 
\begin{eqnarray}
\label{bigl2}
{\mathbb{P}}^{\GG}(R_k^{2}&=&\ell_k^{2})\le C(w,\nn, |V(\GG)|,\ell_0)\bigg[\frac{1}{w(\ell_k^2+\ell_0)}\nonumber\\
&&\,\,\,\,\,\,\,\,\,\,\,\,\,\,\,\,\,\,\,\,\,\,\,\,\,\,\,\,\,\,\,\,\,+\sum_{i=\max\left([k/\nn],\ell_k^2\right)}^{k-\ell_k^2}\frac{1}{w(i+\ell_0)}Q_{\nn-2}\bigg(k-i-\ell_k^2;\ell_0; \infty\bigg)\bigg].
\end{eqnarray}
\end{proposition}
\begin{proof}


\vspace{1mm}

(a) We trivially have for all $m, a, c, j\in\mathbb{N}, m\ge 1$, and for all $b\in\R^+$

\begin{eqnarray*}
\sum_{s=0}^{j}Q_{m}(s+a; b;c)
&\le& \sum_{s=0}^j\sum_{(h^1,\ldots, h^m)\in\mathbb{N}^m\,:\,0\le h^m\le\ldots\le h^1\le \infty\atop h^1+\ldots+h^m=s+a} \frac{1}{\prod_{i=1}^{m} w(b+h^i)}=\sum_{s=0}^{j}Q_{m}(s+a; b;\infty).
\end{eqnarray*}
To show the second part of (\ref{rec2}), we have
\begin{eqnarray*}
\sum_{s=0}^{j}Q_{m}(s+a; b;\infty)&\le& \sum_{s=0}^\infty\sum_{(h^1,\ldots, h^m)\in\mathbb{N}^m\atop h^1+\ldots+h^m=s+a} \frac{1}{\prod_{i=1}^{m} w(b+h^i)}\le \sum_{s=0}^\infty\sum_{(h^1,\ldots, h^m)\in\mathbb{N}^m\atop h^1+\ldots+h^m=s} \frac{1}{\prod_{i=1}^{m} w(b+h^i)}\\
&=&\prod_{i=1}^m\bigg[\sum_{h^i=0}^\infty \frac{1}{w(h^i+b)}\bigg]= (c(b))^m,
\end{eqnarray*}
where for the first inequality we removed the restriction on $Q_m$ that $h^m\le\ldots\le h^1$, and for the second inequality we removed the restriction that $h^1+\ldots+h^m\ge a$.

\vspace{1.5mm}

(b) Fix $k\ge 1$ and $\ell_k^{2}$ such that $0\le\ell_k^{2}\le [k/2]$. 
By means of (\ref{EPfingraph3}) from Proposition \ref{edgegen}, we have
\begin{eqnarray}
\label{eqmain}
{\mathbb{P}}^{\GG}(R_k^{2}=\ell_k^{2})&\le&\sum_{(\ell_k^1,\ell_k^3,\ldots, \ell_k^\nn)\in \mathbb{N}^{\nn-1}: \,0\le \ell_k^\nn\le\ldots\le\ell_k^2\le\ell_k^1\atop \sum_{i=1, i\neq 2}^{\nn}\ell_k^i={k}-\ell_k^2} {\mathbb{P}}^{\GG}(R_k^i=\ell_k^i, i=1,\ldots, \nn)\nonumber\\
& \leq&\frac{|V(\GG)|\cdot\nn!}{w(\ell_0)}
\cdot 
\sum_{(\ell_k^1,\ell_k^3,\ldots, \ell_k^\nn)\in {\mathbb{N}}^{\nn-1}: \,0\le \ell_k^\nn\le\ldots\le\ell_k^2\le\ell_k^1\atop \sum_{i=1, i\neq 2}^{\nn}\ell_k^i={k}-\ell_k^2} \frac{\sum_{i=1}^\nn w(\ell_k^i+\ell_0)}
{\prod_{i=1}^{\nn} w(\ell_k^i+\ell_0)}\nonumber\\
&\le&\frac{|V(\GG)|\cdot\nn!}{w(\ell_0)}
\cdot \bigg(\frac{1}{w(\ell_k^2+\ell_0)}\cdot
 \sum_{i=1\atop i\neq 2}^\nn \sum_{(\ell_k^1,\ell_k^3,\ldots, \ell_k^\nn)\in {\mathbb{N}}^{\nn-1}: \,0\le \ell_k^\nn\le\ldots\le\ell_k^2\le\ell_k^1\atop \sum_{j=1, j\neq 2}^{\nn}\ell_k^j={k}-\ell_k^2}\frac{1}
{\prod_{j=1,j\neq 2,i}^{\nn}\,\,\, w(\ell_k^j+\ell_0)}\nonumber\\
&&+\sum_{(\ell_k^1,\ell_k^3,\ldots, \ell_k^\nn)\in {\mathbb{N}}^{\nn-1}: \,0\le \ell_k^\nn\le\ldots\le\ell_k^2\le\ell_k^1\atop \sum_{j=1, j\neq 2}^{\nn}\ell_k^j={k}-\ell_k^2} \frac{1}{w(\ell_k^1+\ell_0)
\prod_{j=3}^{\nn} w(\ell_k^j+\ell_0)}\bigg).
\end{eqnarray}
For the first inequality in the above, we recall that by Remark (\ref{reducew}) (b) we have $w(\ell_0^e)\le 1$, $e\in E(\GG)$. In the second inequality in (\ref{eqmain}), we split the sum from the first inequality in $\nn-1$ sums which contain $w(\ell_k^2+\ell_0)$ but do not contain one of the other $w(\ell_k^i+\ell_0), 1\le i\le\nn, i\neq 2,$ and the last sum which contains all $w(\ell_k^i+\ell_0), 1\le i\le\nn, i\neq 2,$ but does not contain $w(\ell_k^2+\ell_0)$.

It remains to bound the two sums on the right-hand side of the last inequality in (\ref{eqmain}); in order to do so, we will re-write them in terms of $Q_{\nn-2}$. We focus first on the first double sum term. For $3\le i\le\nn$, all the terms in the sum below do not contain $w(\ell_k^i+\ell_0)$, with $\ell_k^i\le\ell_k^2$. Then 
\begin{eqnarray}
\label{not1}
\lefteqn{\sum_{(\ell_k^1,\ell_k^3,\ldots, \ell_k^\nn)\in {\mathbb{N}}^{\nn-1}: \,0\le \ell_k^\nn\le\ldots\le\ell_k^2\le\ell_k^1\atop \sum_{i=1, i\neq 2}^{\nn}\ell_k^i={k}-\ell_k^2}\frac{1}
{\prod_{j=1,j\neq 2,i}^{\nn}\,\,\, w(\ell_k^j+\ell_0)}}\nonumber\\
&\le&\sum_{\ell_k^i=0}^{\ell_k^2}\,\,\,\,\,\sum_{{(\ell_k^1,\ell_k^{3},\ldots, \ell_k^{i-1})\in {\mathbb{N}}^{i-2}: \,\ell_k^i\le \ell_k^{i-1}\le\ldots\le\ell_k^{3}\le\ell_k^2\le \ell_k^1\atop (\ell_k^{i+1},\ldots, \ell_k^{\nn})\in {\mathbb{N}}^{\nn-i}: \,0\le \ell_k^\nn\le\ldots\le\ell_k^{i+1}\le\ell_k^i}\atop \sum_{j=1,j\neq 2,i}^{\nn}\ell_k^j={k}-\ell_k^2-\ell_k^i} \frac{1}{\prod_{j=1, j\neq 2,i}^{\nn} w(\ell_k^j+\ell_0)}\nonumber\\
&\le& \sum_{\ell_k^i=0}^{\ell_k^2}\,\,\,\,\,\sum_{(h^1,\ldots, h^{\nn-2})\in {\mathbb{N}}^{\nn-2}: \,0\le h^{\nn-2}\le\ldots\le h^{1}\atop \sum_{j=1}^{\nn-2} h^j={k}-\ell_k^2-\ell_k^i} \frac{1}{\prod_{j=1}^{\nn-2} w(h^j+\ell_0)}\nonumber\\
&\le&\sum_{\ell_k^i=0}^{\ell_k^2} Q_{\nn-2}(k-\ell_k^2-\ell_k^i;\ell_0;\infty)\le (c(\ell_0))^{\nn-2}.
\end{eqnarray}
For the first inequality in (\ref{not1}), we split the initial sum in a double sum, first after all the values that $\ell_k^i$ can take, and then after all the values that the remaining $\ell_k^j$ terms, $j\neq 2, i$, can take, given $\ell_k^i$.  For the second inequality, we removed in the inner sum of the first inequality the restriction on the $(\ell_k^{i-1},\ell_k^{i+1})$ that $\ell_k^{i+1}\le\ell_k^i\le\ell_k^{i-1}$, for the third inequality we used the definition of $Q_{\nn-2}$, and for the last inequality we applied (\ref{rec2}).

By a similar argument, we obtain for $i=1$ (this means that all the terms in the sum below contain $w(\ell_k^i+\ell_0), 3\le i\le\nn$, but do not contain $w(\ell_k^1+\ell_0)$)
\begin{eqnarray}
\label{yes1}
\lefteqn{\sum_{(\ell_k^1,\ell_k^3,\ldots, \ell_k^\nn)\in {\mathbb{N}}^{\nn-1}: \,0\le \ell_k^\nn\le\ldots\le\ell_k^2\le\ell_k^1\atop \sum_{i=1, i\neq 2}^{\nn}\ell_k^i={k}-\ell_k^2}\frac{1}
{\prod_{j=3}^{\nn}\,\,\, w(\ell_k^j+\ell_0)}}\nonumber\\
&=&\sum_{s=\max\left([{k}/\nn],\ell_k^2\right)}^{{k}-\ell_k^2}\,\,\sum_{(\ell_k^{3},\ldots, \ell_k^{\nn})\in {\mathbb{N}}^{\nn-2}: \,0\le \ell_k^{\nn}\le\ldots\le\ell_k^{3}\le\ell_k^2\atop \sum_{j=3}^{\nn}\ell_k^j={k}-\ell_k^2-s} \frac{1}{\prod_{j=3}^{\nn}\,\,\, w(\ell_k^j+\ell_0)}\nonumber\\
&=&\sum_{s=\max\left([{k}/\nn],\ell_k^2\right)}^{{k}-\ell_k^2}\,\,\,Q_{\nn-2}({k}-\ell_k^2-s;\ell_0,\ell_k^2)\le (c(\ell_0))^{\nn-2}.
\end{eqnarray}
In the above, we used for the first equality that since $\sum_{i=1}^\nn\ell_k^i=k,$ with $0\le\ell_k^\nn\le\ldots\le\ell_k^2\le\ell_k^1,$ this gives in particular $\nn\ell_k^1\ge k$, from which $\ell_k^1\ge \max([k/\nn],\ell_k^2)$; for the second equality we used the definition of $Q_{\nn-2}$, and for the inequality we used (\ref{rec2}).

Turning now to the 2nd sum term on the right-hand side of (\ref{eqmain}), we have similarly
\begin{eqnarray}
\label{yes2}
\lefteqn{\sum_{(\ell_k^1,\ell_k^3,\ldots, \ell_k^\nn)\in {\mathbb{N}}^{\nn-1}: \,0\le \ell_k^\nn\le\ldots\le\ell_k^2\le\ell_k^1\atop \sum_{j=1, j\neq 2}^{\nn}\ell_k^j={k}-\ell_k^2} \frac{1}{w(\ell_k^1+\ell_0)
\prod_{j=3}^{\nn} w(\ell_k^j+\ell_0)}}\nonumber\\
&=&\sum_{\ell_k^1=\max\left([{k}/\nn],\ell_k^2\right)}^{{k}-\ell_k^2}\frac{1}{w(\ell_k^1+\ell_0)}\sum_{(\ell_k^3,\ldots, \ell_k^\nn)\in {\mathbb{N}}^{\nn-2}: \,0\le \ell_k^\nn\le\ldots\le\ell_k^3\le\ell_k^2\atop \sum_{j=3}^{\nn}\ell_k^j={k}-\ell_k^2-\ell_k^1}\frac{1}
{\prod_{j=3}^{\nn}\,\,\, w(\ell_k^j+\ell_0)}\nonumber\\
&=&\sum_{i=\max\left([{k}/\nn],\ell_k^2\right)}^{{k}-\ell_k^2}\frac{1}{w(i+\ell_0)}Q_{\nn-2}\bigg({k}-i-\ell_k^2;\ell_0;\ell_k^2\bigg).
\end{eqnarray}
By combining (\ref{eqmain}),  (\ref{not1}), (\ref{yes1}) and (\ref{yes2}), we now immediately get (\ref{bigl2}). 
\end{proof}

We are now ready to prove Theorem \ref{mainedgefin}, the main result of this section.

\noindent\textbf{Proof of Theorem \ref{mainedgefin}}
\\
We will show next that $R_\infty^2<\infty$ a.s., which will imply the statement of the theorem. For simplicity and clarity of computations, we will again restrict ourselves to the case with $\ell_0^e\equiv\ell_0\in\R^+,$ $e\in E(\GG)$, the case with general initial weights following by similar arguments by means of (\ref{EPfingraph3}) and of a generalization to Proposition \ref{2orderstat}.

We note first that by Lemma \ref{gexistence}, there exists under assumption (\ref{Azero}) a function $g\colon \bN \rightarrow \left[0, \infty\right)$ such that (i) $g\left(\cdot\right)$ is increasing and $\textstyle \lim_{\ell \uparrow \infty}g\left(\ell\right)=\infty$;  and (ii) $\textstyle M:=\sum _{\ell\geq 1}^{\infty}\frac{g\left(\ell+\ell_0\right)}{w(\ell+\ell_0)} < \infty$. Therefore, since g is an increasing function, we have by the monotone convergence theorem and in view of Proposition \ref{2orderstat} (b), that
\begin{eqnarray}
\label{ineq1}
{\mathbb{E}}^{\GG}(g(R_\infty^2))&=&\lim_{k\rightarrow\infty}{\mathbb{E}}^{\GG}(g(R_k^2))=\lim_{k\rightarrow\infty}\sum_{\ell_k^2=0}^{[{k}/2]} g(\ell_k^2){\mathbb{P}}^{\GG}(R_k^2=\ell_k^2)\nonumber\\
&\le&C(w,\nn, |V(\GG)|,\ell_0)\lim_{k\rightarrow\infty}\sum_{\ell_k^2=0}^{[{k}/2]} g(\ell_k^2+\ell_0)\bigg[\frac{1}{w(\ell_k^2+\ell_0)}\nonumber\\
&&+\sum_{i=\max\left(\left[{k}/\nn\right],\ell_k^2\right)}^{{k}-\ell_k^2}\frac{1}{w(i+\ell_0)}Q_{\nn-2}\bigg({k}-i-\ell_k^2;\ell_0;\infty\bigg)\bigg],
\end{eqnarray}
for some ${C}(w,\nn, |V(\GG)|,\ell_0)>0$. By property (ii) of $g$, we have for all ${k}\ge 1$ 
\begin{equation}
\label{ineq2}
\sum_{\ell_k^2=0}^{\infty} \frac{g(\ell_k^2+\ell_0)}{w(\ell_k^2+\ell_0)}<M<\infty.
\end{equation}
To bound the second sum in (\ref{ineq1}), we have
\begin{eqnarray}
\label{ineq3}
\lefteqn{\sum_{\ell_k^2=0}^{[{k}/2]} g(\ell_k^2+\ell_0)\sum_{i=\max\left([{k}/\nn],\ell_k^2\right)}^{{k}-\ell_k^2}\frac{1}{w(i+\ell_0)}Q_{\nn-2}\bigg({k}-i-\ell_k^2;\ell_0;\infty\bigg)}\nonumber\\
&\le&\sum_{\ell_k^2=0}^{[{k}/2]}\,\,\, \sum_{i=\ell_k^2}^{{k}-\ell_k^2}\frac{g(\ell_k^2+\ell_0)}{w(i+\ell_0)}Q_{\nn-2}\big({k}-i-\ell_k^2;\ell_0;\infty\big)\le \sum_{i=0}^k\sum_{\ell_k^2=0}^{\min(i,k-i)}\frac{g(\ell_k^2+\ell_0)}{w(i+\ell_0)}Q_{\nn-2}\big({k}-i-\ell_k^2;\ell_0;\infty\big)\nonumber\\
&\le&\sum_{i=0}^k \frac{g(i+\ell_0)}{w(i+\ell_0)}\sum_{\ell_k^2=0}^{\min(i,k-i)}Q_{\nn-2}\big({k}-i-\ell_k^2;\ell_0;\infty\big)\le (c(\ell_0))^{\nn-2}\sum_{i=0}^{{k}}\frac{g(i+\ell_0)}{w(i+\ell_0)}\nonumber\\
&\le&( c(\ell_0))^{\nn-2}M,
\end{eqnarray}
where $c(\ell_0):=\sum_{i=0}^\infty 1/{w(\ell_0+i)}$. For the first inequality in (\ref{ineq3}), we removed the restriction that $i\ge [k/\nn]$, for the second inequality we changed the summation order between $\ell_k^2$ and $i$. For the third inequality  we used that $g$ is an increasing function and that $\ell_k^2\le i$ in order to take the $g$ terms out of the inner sums,  for the fourth inequality we used (\ref{rec2}) to eliminate the sums over the $Q_{\nn-2}$ terms, and for the last inequality we used property (ii) of $g$.

It follows now from (\ref{ineq1}), (\ref{ineq2}) and (\ref{ineq3}) that ${\mathbb{E}}^{\GG}(g(R_\infty^2))<\infty$, which implies that a.s. $g(R_\infty^2)<\infty$. Combining this with the fact that $g$ is increasing and $\lim_{\ell\uparrow\infty} g(\ell)=\infty$, we get that a.s. $R_\infty^2<\infty$.
\qed
\begin{remark}
\label{difw}
For arbitrary finite graphs $\GG$, our arguments above can be extended to more general classes of reinforcement. We will give below just two such examples.
\begin{itemize}
\item [(a)] Firstly, even though computationally more intensive, we can easily extend the result in Theorem \ref{mainedgefin} by the above reasoning to the case where each edge $e\in E(\GG)$ has its own weight function $w_e$ satisfying
$$\sum_{\ell=1}^\infty\frac{1}{w_e(\ell+\ell_0^e)}<\infty,\forall e\in E(\GG).$$
More precisely, one first obtains similar bounds to the ones in Proposition \ref{2orderstat}, but depending now on each of the $w_e$. The key point is that, under the above condition, there exists by Lemma \ref{gexistencem} an increasing function $g$, common to all the $w_e$, satisfying the assumptions in the lemma. This allows us to extend the arguments from Theorem \ref{mainedgefin} to this more general setting.
\item [(b)] Assume now that, in the beginning, the walk has the reinforcement function $w_0:\mathbb{N}\rightarrow\R^+$, satisfying (\ref{Azero}) and (\ref{initialweight}). Assume also that the reinforcement function is replaced a countable number of times (at the possibly random times $j_1, j_2,\ldots$), by the new reinforcement functions $w_1, w_2,\ldots:\mathbb{N}\rightarrow\R^+$, all satisfying (\ref{Azero}) and (\ref{initialweight}). Then if there exists $C>0, C^\nn\max_v deg (v)\le 1$, such that for all $i=1,2,\ldots$ 
\begin{equation}
\label{randweight}
\sup_{\ell}\frac{w_i(\ell)}{w_{i-1}(\ell)}\le C,
\end{equation}
we have
\begin{equation}
\label{random1}
{\mathbb{P}}^{\GG}(\GG_\infty~\mbox{has only one edge}) = 1.
\end{equation}
Note that (\ref{randweight}) is fulfilled by a large number of sub- and super-exponential functions, in particular by $w_i(k)=k^{\rho_i}, \rho_0\ge\ldots\ge\rho_m\ldots>1, i=0,1\ldots.$

\begin{proof}
By means of (\ref{randweight}) and of an argument of induction after $k$, we can obtain for all $j_{i-1}\le k\le j_{i}, i=0,1,\ldots,$ a similar inequality to (\ref{EPfingraph1}), but in function of $w_{i-1}$ terms. We can then show (\ref{random1}) by the same arguments as in the proof of Theorem \ref{mainedgefin} above.
\end{proof}
\end{itemize}
\end{remark}

\subsection{Attraction to one edge on infinite connected graphs of bounded degree}
\label{InfGraphs}
In this section we will prove Theorem \ref{maininfedge}. 
In preparation for the proof of the theorem, we will first need to show Lemmas \ref{Lem:Stuck_new_edge} and \ref{Lem:Stuck_n_distance} below, which give us upper bounds on the probability that the walk ever visits more than $n$ vertices. Even though such results as in these two lemmas are known in the literature, we will provide here a proof for completeness purposes. We also note here that due to the permutations in Proposition \ref{edgegen}, which we are only able to remove in the special case of a star graph, we cannot apply the same arguments as in Theorem \ref{mainedgefin} directly to infinite graphs,  as the bounds would blow up with $n$. The lemmas below help us to bypass this problem.


Let us denote by $\norm{\cdot}$ the graph distance in $\GG$ when the graph is centred around $v_0$, where we recall that $v_0$ is the initial position of the walk. 
Let us define for each $n$ the stopping time 
\begin{equation*}
T_n:=\inf\{k \geq 1\colon \norm{I_k}=n\},
\end{equation*} where $T_n$ is taken to be infinite otherwise. We define for each $n \geq 1$
\[V_n:=\left\{v\in V(\GG)\colon \, \norm{v}=n, \text{ there exists } v'\in V(\GG) \text{ such that } \{v, v'\} \in \NN_v \text{ with } \norm{v'}=n+1\right\}.
\] 
Observe that since $\GG$ is an infinite connected graph, the set $V_n$ is non-empty for every $n \geq 1$. Consider the stopping time 
\[T_{V_n}:=\inf\{k \geq 1\colon I_k \in V_n\},
\] with the convention that $T_{V_n}=\infty$ if $(I_k)_{k\ge 1}$ never enters $V_n$. Observe that for each $n \geq 1$, $T_n \leq T_{V_n}< T_{n+1}$. Let us denote the event
\begin{multline*}
B_{V_n}:=\big\{\mbox{At time}~T_{V_n} \mbox{the random walk gets in contact with an edge it has not been in contact} \\ \mbox{with before and then it immediately gets stuck on this edge}\big\}.
\end{multline*}
Both above and in what follows, the walk is said to come for the first time in contact with an edge at time $k\ge 0$ if it visits for the first time at time $k$ either of the two vertices of the respective edge.
\begin{Lemma}\label{Lem:Stuck_new_edge}
Let $\GG$ be an infinite connected graph of bounded degree. If $w$ satisfies (\ref{Azero}) and (\ref{initialweight}), then there exists $p>0$, depending only on $w\left(\cdot\right)$, the maximal degree of the graph $D(\GG)$ and the initial configuration of weights $\ell_0^e,\, e \in E(\GG)$, such that for all $n\ge 1$ we have
\begin{equation}\label{Eq: Stuck_new_edge}
{\mathbb{P}}^{\GG} \left(B_{V_n}\big{\vert}T_{V_n} < \infty\right)\geq p.
\end{equation}
\end{Lemma}
\begin{proof} Fix $n\ge 1$ arbitrarily. Before we proceed, we note that 
$${\mathbb{P}}^{\GG} \left(T_{V_n} < \infty \right)\ge {\mathbb{P}}^{\GG} \left(T_{n+1} < \infty \right)\ge {\mathbb{P}}^{\GG} \left(T_{n+1} =n+1 \right)>0.$$

On the event $\{T_{V_n}=k\}$, we have that $I_k=v$ for some $v \in V_n$. Firstly, observe that there exists exactly one $e \in \NN_v$ such that $X_t^{e}=\ell_0^e$, for $t \leq k-1$, and $X_k^{e}=\ell_0^e +1$, and for every $f\neq e \in \NN_v$, we must have $X_k^{f}=\ell_0^f$ (or else $v$ would have been visited before). Moreover, there exists a $v' \in  V(\GG)$, such that $\left\{v, v'\right\}\in \NN_v$ and $\norm{v'}=n+1$, and $v'$ has never been visited before. This implies that $\left\{v, v'\right\}$ is a new edge that the walk has never crossed up to time $k$, which the walk comes in contact with for the first time at time $k$. Due to the assumptions on the weight function $w(\cdot)$, there exists a $p$ uniformly bounded away from $0$, and depending only on $w\left(\cdot\right)$,  $D(\GG)$ and on the initial configuration of weights $\ell_0^e,\, e \in E(\GG)$ (in fact, only on weights $\ell_0^f$ and $\ell_0^{f'}$, where $f$ and $f'$ are the edges incident on $v$ and $v'$ respectively) such that the walk keeps traversing solely the edge $\{v,v'\}$ after time $k$.
Therefore, in symbols, on the event $\{T_{V_n}=k\}$ we have  
\begin{eqnarray}\label{Eq: Stuck_new_edge_v}
\lefteqn{{\mathbb{P}}^{\GG}\left(\{I_t, I_{t+1}\}=\{v, v'\}, \text{ for all } t \geq k\big{\vert} {\cal {F}}_k\right)}\nonumber\\
&\ge&\prod_{i=0}^\infty\frac{\left(w(i+\ell_0^{\{v,v'\}})\right)^2}{\big[w(i+\ell_0^{\{v,v'\}})+\sum_{f\in {\cal{N}}(v)\setminus{\{v,v'\}}}w(\ell_0^f)\big]\big[{w(i+\ell_0^{\{v,v'\}})+\sum_{f'\in {\cal{N}}(v')\setminus {\{v,v'\}}}w(\ell_0^{f'})}\big]}\nonumber\\
&\ge&\prod_{i=0}^\infty\frac{\left(w(i+\ell_0^{\{v,v'\}})\right)^2}{\big[w(i+\ell_0^{\{v,v'\}})+D(\GG) \sup_{e\in E(\GG)}w(\ell_0^e)\big]^2}=\prod_{i=0}^\infty\bigg(\frac{w(i+\ell_0^{\{v,v'\}})+D(\GG) \sup_{e\in E(\GG)}w(\ell_0^e)}{w(i+\ell_0^{\{v,v'\}})}\bigg)^{-2}\nonumber\\
&=&\prod_{i=0}^{\infty}\bigg(1+\frac{D(\GG) \sup_{e\in E(\GG)}w(\ell_0^e)}{w(i+\ell_0^{\{v,v'\}})}\bigg)^{-2}\ge\exp\bigg(-2\sum_{i=1}^\infty\frac{D(\GG) \sup_{e\in E(\GG)}w(\ell_0^e)}{w(i+\ell_0^{\{v,v'\}})}\bigg)\nonumber\\
 &=&:p>0.
\end{eqnarray}
In the above, we have used for the second inequality (\ref{initialweight}), and for the third inequality we used $1+x\le e^{x}, x\in\mathbb{R},$ and (\ref{Azero}). Equation (\ref{Eq: Stuck_new_edge_v}) implies that 
$${\mathbb{P}}^{\GG}\left(B_{V_n}|T_{V_n}=k\right)\geq p,$$
from which we get
\begin{eqnarray}\label{Eq: Stuck_new_edge_proof}
{\mathbb{P}}^{\GG} \left(B_{V_n}\big{\vert}T_{V_n}< \infty\right)&=&\frac{{\mathbb{P}}^{\GG} \left(B_{V_n}\cap\big (T_{V_n}< \infty)\right)}{{\mathbb{P}}^{\GG} \left(T_{V_n}< \infty\right)}=\sum_{k\ge n}\frac{{\mathbb{P}}^{\GG} \left(B_{V_n}\cap\big (T_{V_n}=k)\right)}{{\mathbb{P}}^{\GG} \left(T_{V_n}< \infty\right)}\nonumber\\
&=&\sum_{k\ge n}\frac{{\mathbb{P}}^{\GG} \left(B_{V_n}\big{\vert} T_{V_n}=k\right) {\mathbb{P}}^{\GG} \left(T_{V_n}=k\right)}{{\mathbb{P}}^{\GG} \left(T_{V_n}< \infty\right)}\geq \sum_{k\ge n}\frac{p\, {\mathbb{P}}^{\GG} \left(T_{V_n}=k\right)}{{\mathbb{P}}^{\GG} \left(T_{V_n}< \infty\right)}=p.
\end{eqnarray}
\end{proof}
\begin{Lemma}\label{Lem:Stuck_n_distance}
Let $\GG$ be an infinite connected graph of bounded degree. If $w$ satisfies (\ref{Azero}) and (\ref{initialweight}), then for every $n \geq 1$,
\begin{equation}\label{Eq:Stuck_k_distance}
{\mathbb{P}}^{\GG}\bigg(\displaystyle \sup_{k \geq 1} \norm{I_k} > n\bigg)\leq \left(1- p\right)^{[n/2]},
\end{equation} 
where $p$ is as in \eqref{Eq: Stuck_new_edge}.
\end{Lemma}
\begin{proof} Recall from the proof of Lemma \ref{Lem:Stuck_new_edge} that ${\mathbb{P}}^{\GG} \left(T_{V_n} < \infty \right)>0, n\ge 1$. 

The proof follows by induction on $n \geq 1$. Trivially, we have ${\mathbb{P}}^{\GG}\bigg(\displaystyle \sup_{k \geq 1} \norm{I_k} > 1\bigg)\le 1$.
Observe now that $\{\textstyle \sup_{k \geq 1} \norm{I_k} > 2 \}$ implies that $T_{V_1}< \infty$, and also that $B^c_{V_1}$ must happen (otherwise the walk will get stuck immediately after time $T_{V_1}$ on a new edge it comes in contact with). 
Therefore, 
\begin{eqnarray} \label{Eq: Base_step_1}
\nonumber {\mathbb{P}}^{\GG}\bigg(\displaystyle \sup_{k \geq 1} \norm{I_k} > 2\bigg) &\leq &{\mathbb{P}}^{\GG}\left(B^c_{V_1}\cap (T_{V_1}< \infty)\right) =  {\mathbb{P}}^{\GG}\left(B^c_{V_1}\Big{\vert}T_{V_1}< \infty\right) {\mathbb{P}}^{\GG}\left(T_{V_1}< \infty\right) \\
&\le & (1-p),
\end{eqnarray}
where for the second inequality we used Lemma \ref{Lem:Stuck_new_edge}. By arguments similar to the ones above, we have 
\begin{eqnarray}
\label{Eq: Base_step_2}
{\mathbb{P}}^{\GG}\left(\displaystyle \sup_{k \geq 1} \norm{I_k} > 3\right) &\leq &  {\mathbb{P}}^{\GG}\left(B^c_{V_2}\cap (T_{V_2}< \infty)\right)
\le (1-p).
\end{eqnarray}
Thus, \eqref{Eq: Base_step_1} and \eqref{Eq: Base_step_2} together prove the base step of the induction. Let us assume now that \eqref{Eq:Stuck_k_distance} holds for all $m \leq 2n+1$. To complete the proof, we need to show
 that \eqref{Eq:Stuck_k_distance} holds for $m=2n+2$ and for $m=2n+3$. 

Observe now that $\{\textstyle \sup_{k \geq 1} \norm{I_k} > 2n+2\} =\{\sup_{k \geq 1} \| I_k\| \geq 2n+3\}$, which in turn implies that $T_{2n+2}< \infty$. This gives that $T_{V_{2n+1}}<T_{2n+2}< \infty$, and that also the event $B^c_{V_{2n+1}}$ must happen (otherwise the walk will get stuck immediately after time $T_{V_{2n+1}}$ on a new edge it comes in contact with). 
Therefore
\begin{eqnarray*}
{\mathbb{P}}^{\GG}\bigg(\displaystyle \sup_{k \geq 1} \norm{I_k} > 2n+2\bigg)\!\!\!&\leq &\!\!\!{\mathbb{P}}^{\GG}\left(B^c_{V_{2n+1}}\cap (T_{V_{2n+1}}< \infty)\right) 
\le {\mathbb{P}}^{\GG}\left(B^c_{V_{2n+1}}\Big{\vert}T_{V_{2n+1}}< \infty\right) \!{\mathbb{P}}^{\GG}\left(T_{{2n+1}}< \infty\right)\nonumber\\
 &\le& \left(1- p\right) {\mathbb{P}}^{\GG}\bigg(\displaystyle \sup_{k \geq 1} \norm{I_k} >2n\bigg)\le (1-p)^{n+1},
\end{eqnarray*}
where for the second inequality we used $\{T_{V_{2n+1}}< \infty\}\subseteq \{T_{{2n+1}}< \infty\}$, for the third inequality we used Lemma \ref{Lem:Stuck_new_edge} and the fact that $\big\{T_{{2n+1}}< \infty\big\}\subseteq \bigg\{\displaystyle \sup_{k \geq 1} \norm{I_k} >2n\bigg\}$, and for the fourth inequality we used our induction hypothesis.

For $m=2n+3$, by conditioning on the event $\{T_{V_{2n+2}}< \infty\}$, we can use the same arguments with $B^c_{V_{2n+1}}$ replaced by $B^c_{V_{2n+2}}$, to obtain 
\begin{equation*}
{\mathbb{P}}^{\GG}\bigg(\displaystyle \sup_{k \geq 1} \norm{I_k} > 2n+3\bigg) \leq (1-p)^{n+1}.
\end{equation*}
\end{proof}
\begin{remark} The statement of Lemma \ref{Lem:Stuck_n_distance} is a precise restatement of Lemma 25 of Cotar and Limic \cite{CotarLimic}. However, in Lemma 25 from \cite{CotarLimic} the authors use in the proof the assumption that the edge-reinforced random walk gets stuck on one edge, whereas we have considered no such assumption in our proof above. Moreover,  our proof adapts easily to the the more general types of reinforced processes studied in Remark \ref{difw} (a) and (b), provided that the equivalent inequality to (\ref{Eq: Stuck_new_edge_v}) holds.
\end{remark}

\noindent We will prove next the main result of this section. 

\noindent\textbf{Proof of Theorem \ref{maininfedge}}

\textbf{Step 1:} For all $n\ge 1$, let $\GG_n$ be the graph centred at $v_0$ formed of the vertices in $\GG$ at graph distance $\le n$, and let $\partial\GG_n$ be the boundary of $\GG$, that is, the graph formed of vertices at graph distance $n+1$ which are connected by an edge to vertices in $\GG_n$. Since $\GG$ is an infinite graph, the set $\partial\GG_n$ is non-empty for every $n\ge 1$. We denote by
$$A_n= \{ |\GG_\infty|=1, ~~\mbox{the walk never leaves}~\GG_n\}.$$
We have by definition of $A_n$ that
$$A_n \subseteq A_{n+1}, \forall n\ge 1,~\mbox{and}~ \{|\GG_\infty|=1\}\supseteq \cup_{n\ge 1}A_n.$$
Therefore we have
$${\mathbb{P}}^{\GG}(|\GG_\infty|=1)\ge\lim_{n\rightarrow\infty} {\mathbb{P}}^{\GG}( A_n).$$
From Lemma \ref{Lem:Stuck_n_distance} we have $\mathbb{P}^{\GG}(\mbox{the walk never leaves}~ \GG_n)\ge 1-\theta^{[n/2]}>0$, for some fixed $0<\theta<1$ which depends only on $w, \cal{D}(\GG)$ and the initial configurations of weights $\ell_0^e,e\in E(\GG)$. Then
\begin{eqnarray}
\label{bd1}
{\mathbb{P}}^{\GG}( A_n)&=&{\mathbb{P}}^{\GG}(\mbox{the walk never leaves}~\GG_n) {\mathbb{P}}^{\GG}(|\GG_\infty|=1~ |~\mbox{the walk never leaves}~ \GG_n)\nonumber\\
&\ge& (1-\theta^{[n/2]}) {\mathbb{P}}^{\GG}(|\GG_\infty|=1~ | ~\mbox{the walk never leaves}~\GG_n).
\end{eqnarray}

\textbf{Step 2:} We will show here that for each $n\ge 1$, we have
\begin{equation}
\label{bd2}
{\mathbb{P}}^{\GG}(|\GG_\infty|=1~ |~\mbox{the walk never leaves}~\GG_n)=1.
\end{equation}

\vspace{2mm}

\noindent Coupled with (\ref{bd1}), (\ref{bd2}) will imply that ${\mathbb{P}}^{\GG}( A_n)\ge 1-\theta^{[n/2]}$; since ${\mathbb{P}}^{\GG}(|\GG_\infty|=1)\ge\lim_{n\rightarrow\infty} {\mathbb{P}}^{\GG}( A_n)$, the statement of the theorem will follow immediately.

\noindent In order to prove (\ref{bd2}), we introduce the probability measure
$${\mathbb{P}}_{|\GG_n}(\cdot):={\mathbb{P}}^{\GG}(\cdot ~|~\mbox{the walk never leaves}~\GG_n).$$
Then (\ref{bd2}) is equivalent to ${\mathbb{P}}_{|\GG_n}(|\GG_\infty|=1)=1$, which last one we will proceed to show as in the proof of Theorem \ref{mainedgefin}. We will start by finding estimates for the order statistics for the number of edge traversals. Since on ${\mathbb{P}}_{|\GG_n}$ only edges in $\GG_n$ are traversed, we need only consider the corresponding order statistics for edges in $E(\GG_n)$,  where as before, $E(\GG_n)$ is the set of edges in $\GG_n$ and $V(\GG_n)$ is the set of vertices in $\GG_n$.  This reduces to finding, for $v\in V(\GG_n)$ and $k\geq 1$,  upper bounds for ${\mathbb{P}}_{|\GG_n}(X_k^e-\ell_0^e=\ell_k^e, e \in E(\GG_n), A_{v,k})$, where $\ell_k^e\in\mathbb{N}$, $e\in E(\GG_n)$, with $\sum_{e \in E({\scriptsize \GG_n})}\ell_k^e=k$, whereas $X_k^e=\ell_0^e$ for all $e\in E(\GG)\setminus E(\GG_n)$. More precisely, we will show in Step 3 below that for all $v\in V(\GG_n)$
\begin{eqnarray}
\label{condgn1a}
\lefteqn{{\mathbb{P}}_{|\GG_n}(X_k^e-\ell_0^e=\ell_k^e, e \in E(\GG_n), A_{v,k})}\nonumber\\
&=&\frac{{\mathbb{P}}^{\GG}(X_k^e-\ell_0^e=\ell_k^e, e \in E(\GG_n), A_{v,k}, ~\mbox{the walk never leaves}~ \GG_n)}{{\mathbb{P}}^{\GG}(~\mbox{the walk never leaves}~ \GG_n)}\nonumber\\ 
&\leq& 
\frac{1}{{\mathbb{P}}^{\GG}(~\mbox{the walk never leaves}~ \GG_n)}\,\,\,\cdot\frac{\prod_{e\in E({\scriptsize \GG_n})}w(\ell_0^e)}{\min_{e \in E(\GG_n)}\, w(\ell_0^e) }
\,\,\cdot 
\frac{\sum_{e\in \NNs_v,  e\in\GG_n}w(\ell_k^e+\ell_0^e)}
{\prod_{e\in E({\scriptsize \GG_n})} w(\ell_k^e+\ell_0^e)},
\end{eqnarray}
where $\mathbb{P}^{\GG}(\mbox{the walk never leaves}~ \GG_n)>0$.  Given (\ref{condgn1a}), the corresponding inequality for ${\mathbb{P}}_{|\GG_n}(R_k^i=\ell_k^i, i=1,\ldots, \nn, A_{v,k})$, $\nn=|E(\GG_n)|$, follows immediately. This allows us to get upper bounds for ${\mathbb{P}}_{|\GG_n}(R_k^{2}=\ell_k^{2})$ as in Proposition \ref{2orderstat} above. Then, by the same arguments as in Theorem \ref{mainedgefin}, we can show that ${\mathbb{E}}_{|\GG_n}(g(R_\infty^2))<\infty$, from which ${\mathbb{P}}_{|\GG_n}$ a.s. we have $R_\infty^2<\infty$. Since on ${\mathbb{P}}_{|\GG_n}$ the walk only visits finitely many vertices, this gives 
${\mathbb{P}}_{|\GG_n}(|\GG_\infty|=1)=1$.


\vspace{1.5mm}

\textbf{Step 3:} It remains to show here (\ref{condgn1a}) for all $v\in V(\GG_n)$. Even though we borrow the \textit{key} induction idea from the proof of Proposition 1.19 in \cite{CotarLimic} to show this, the argument is more delicate now due to the fact that we restrict our walk to stay in $\GG_n$.  This means in particular that one needs to take a lot more care when setting up a suitable recursion formula (on $k\ge 1$) to use for the induction argument, which recursion is the main ingredient in the proof by induction used in Proposition 1.19 from \cite{CotarLimic}. Our first main goal below will thus be to show the recursion formula from (\ref{condexp}) below.

To begin with, we note that for all $k\ge 1$, we have
\begin{multline}
\label{condk0}
{\mathbb{P}}^{\GG}(X_k^e-\ell^e_0=\ell_k^e,\, e\in E(\GG_n), 
A_{v,k}, ~\mbox{the walk never leaves}~ \GG_n)\\
\le {\mathbb{P}}^{\GG}(X_k^e-\ell^e_0=\ell_k^e,\, e\in E(\GG), 
A_{v,k}, ~\mbox{the walk does not leave}~ \GG_n~\mbox{by time}~k).
\end{multline}
 Next,  we will show that for all $v\in V(\GG_n)$ and for all $k\ge 1$, we have
\begin{multline}
\label{constrgn}
{\mathbb{P}}^{\GG}(X_k^e-\ell^e_0=\ell_k^e,\, e\in E(\GG_n), A_{v,k}, ~\mbox{the walk does not leave}~ \GG_n~\mbox{by time}~k)\\
\le\frac{\prod_{e\in E({\scriptsize \GG_n})}w(\ell_0^e)}{\min_{e \in E(\GG_n)}\, w(\ell_0^e) }\,\,\cdot\frac{\sum_{e\in \NNs_v,  e\in\GG_n}w(\ell_k^e+\ell_0^e)}
{\prod_{e\in E({\scriptsize \GG_n})} w(\ell_k^e+\ell_0^e)},
\end{multline}
which, combined with (\ref{condk0}), will give (\ref{condgn1a}). In order for the event 
$\{X_k^e-\ell^e_0=\ell_k^e, e \in E(\GG_n), A_{v,k}\}$
to happen, it must be 
$I_{k-1}=v_i$ for some
$v_i\sim v$ 
such that $\ell_k^{\{v,v_i\}}\ge 1$, and 
furthermore it must be $\{I_{k-1},I_k\}=\{v_i,v\}$. Before we proceed, let below for simplicity of exposition for $v\in V(\GG_n)$
$$B_{v,k}\big((\ell_k^e)_{e\in E(\GG_n)}\big):=\big\{X_k^e-\ell^e_0=\ell_k^e, e \in E(\GG_n), A_{v,k}\big\},~~\Theta_{n,k}:=\{\mbox{the walk does not leave}~ \GG_n~\mbox{by time}~k\}.$$
Then we have on the left-hand side of (\ref{constrgn}) for $v\in V(\GG_n)$ and $k\ge 1$ 
\begin{eqnarray}
\label{condk}
\lefteqn{{\mathbb{P}}^{\GG} \big(B_{v,k}\big((\ell_k^e)_{e\in E(\GG_n)}\big)\cap\Theta_{n,k}\big)}\nonumber\\
&=&\sum_{i=1:\,\ell_k^{\{v,v_i\}}\ge 1}^{d_v} {\mathbb{P}}^{\GG}\big(B_{v_i,k-1}\big(\ell_k^{\{v_i,v\}}-1,(\ell_k^e)_{e\in E(\GG_n)\setminus \{v_i,v\}}\big) \cap \{\{I_{k-1},I_k\}=\{v_i,v\}\}\cap\Theta_{n,k}\big)\nonumber\\
&=&\sum_{i=1: \,\ell_k^{\{v,v_i\}}\ge 1, v_i\in V(\GG_n)}^{d_v} {\mathbb{P}}^{\GG}\big(B_{v_i,k-1}\big(\ell_k^{\{v_i,v\}}-1,(\ell_k^e)_{e\in E(\GG_n)\setminus \{v_i,v\}}\big)\cap \{\{I_{k-1},I_k\}=\{v_i,v\}\} \cap\Theta_{n,k}\big)\nonumber\\
&\le&\sum_{i=1: \,\ell_k^{\{v,v_i\}}\ge 1, v_i\in V(\GG_n)}^{d_v} {\mathbb{P}}^{\GG}\big(B_{v_i,k-1}\big(\ell_k^{\{v_i,v\}}-1,(\ell_k^e)_{e\in E(\GG_n)\setminus \{v_i,v\}}\big)\cap \{\{I_{k-1},I_k\}=\{v_i,v\}\} \cap\Theta_{n,k-1}\big)\nonumber\\
\end{eqnarray}
where for the second equality in (\ref{condk}) we restricted the summation only to $v_i\in V(\GG_n)$, and for the inequality we replaced the event $\Theta_{n,k}$ by the bigger event $\Theta_{n,k-1}$ in order to obtain a recursion in $k$ on the quantity on the left-hand side of (\ref{constrgn}).
Expanding (\ref{condk}) further, we have for $v\in V(\GG_n)$
\begin{eqnarray}
\label{condexp}
\lefteqn{{\mathbb{P}}^{\GG}\big(B_{v,k}\big((\ell_k^e)_{e\in E(\GG_n)}\big)\cap\Theta_{n,k}\big)}\nonumber\\
&\le&\sum_{i=1: \,v_i\in V(\GG_n),\ell_k^{\{v,v_i\}}\ge 1}^{d_v}\!\!\!\!\!{\mathbb{P}}^{\GG}\big(B_{v_i,k-1}\big(\ell_k^{\{v_i,v\}}-1,(\ell_k^e)_{e\in E(\GG_n)\setminus \{v_i,v\}}\big)\cap ~\Theta_{n,k-1}\big)\times\nonumber\\
&&\,\,\,\,\,\,\,\,\,\,\,\,\,\,\,\,\,\,\,\,\,\,\,\,\,\,\,\,\,\,\,\,\,\,\,\,\,\,\,\,\,\,\,\,\,{\mathbb{P}}^{\GG}\!\big(\{I_{k-1},I_k\}\!=\!\{v_i,v\}\big{\vert} B_{v_i,k-1}\big(\ell_k^{\{v_i,v\}}-1,(\ell_k^e)_{e\in E(\GG_n)\setminus \{v_i,v\}}\big)\cap ~\Theta_{n,k-1}\big)\nonumber\\
&=&\sum_{i=1:\, v_i\in V(\GG_n)\atop \,\,\,\ell_k^{\{v,v_i\}}\ge 1}^{d_v}\!\!\!\!\!\! {{\mathbb{P}}^{\GG}}\big(B_{v_i,k-1}\big(\ell_k^{\{v_i,v\}}-1,(\ell_k^e)_{e\in E(\GG_n)\setminus \{v_i,v\}}\big)\cap ~\Theta_{n,k-1}\big)\times\nonumber\\
&&\,\,\,\,\,\,\,\,\,\,\,\,\,\,\,\,\,\,\,\,\,\,\,\,\,\,\,\,\,\,\,\,\,\,\,\,\,\,\,\,\,\,\,\,\,\frac{w(\ell_k^{\{v,v_i\}}-1+\ell_0^e)}{w(\ell_k^{\{v,v_i\}}-1+\ell_0^e) +\sum_{e\in\GG_n\cap \NNs_{v_i}\atop e\neq \{v,v_i\}}w(\ell_k^e+\ell_0^e)+\sum_{e\in\partial\GG_n\cap \NNs_{v_i}}w(\ell_0^e)},
\end{eqnarray}
where we recall that $\NN_{v}:=\{e_1^{v}, e_2^{v},\ldots e_{d_{v}}^{v}\}$ is the set of  edges incident to $v$. For the inequality in (\ref{condexp}) we used conditioning in the last equation in (\ref{condk}), and for the equality we used  that by (\ref{ruleerrw})
\begin{multline*}
{\mathbb{P}}^{\GG}\!\big(\{I_{k-1},I_k\}\!=\!\{v_i,v\}\big{\vert}  B_{v_i,k-1}\big(\ell_k^{\{v_i,v\}}-1,(\ell_k^e)_{e\in E(\GG_n)\setminus \{v_i,v\}}\big)\cap ~\Theta_{n,k-1}\big)\\
=\frac{w(\ell_k^{\{v,v_i\}}-1+\ell_0^e)}{w(\ell_k^{\{v,v_i\}}-1+\ell_0^e) +\sum_{e\in\GG_n\cap \NNs_{v_i}\atop e\neq \{v,v_i\}}w(\ell_k^e+\ell_0^e)+\sum_{e\in\partial\GG_n\cap \NNs_{v_i}}w(\ell_0^e)}.
\end{multline*}
The use of (\ref{ruleerrw}) in (\ref{condexp}) above is in fact the main reason that we developed our recursion formula on the quantity on the left-hand side of (\ref{constrgn}) rather than on the quantity on the left-hand side of (\ref{condk0}).

\vspace{1.5mm}

Given (\ref{condexp}), we can now prove (\ref{constrgn}) by using an argument of induction on $k\ge 1$ similar to the one in Proposition 19 from \cite{CotarLimic} (see also Proposition \ref{Pfingraph1v} below for another example of a related induction argument in the case of VRRW). Formally, we argue as follows. 

The base of induction at $k=0$, that is
\begin{multline}
\label{baseind}
{\mathbb{P}}^{\GG}(X_0^e=\ell^e_0,\, e\in E(\GG_n), A_{v,0}, ~\mbox{the walk does not leave}~ \GG_n~\mbox{by time}~0)\le\frac{\sum_{e\in \NNs_v,  e\in\GG_n}w(\ell_0^e)}{\min_{e \in E(\GG_n)}\, w(\ell_0^e) },
\end{multline}
clearly holds, since when the left-hand-side is $0$
the right-hand-side is positive, and when
the left-hand-side is $1$ (i.e., when $v=v_0$) the right-hand-side is greater or equal than $1$.

Let us assume now that (\ref{constrgn}) holds for 
all $i\le k-1$ and $\ell_i^e\in\mathbb{N}$, $e\in E(\GG_n)$, with $\sum_{e \in E({\scriptsize \GG_n})}\ell_i^e=i$. 
For the induction step 
we need to show that the bound holds for $i=k$ and $\ell_k^e\in\mathbb{N}$, $e\in E(\GG_n)$, with $\sum_{e \in E({\scriptsize \GG_n})}\ell_k^e=k$. By means of (\ref{condexp}), (\ref{baseind}) and of the induction hypothesis, we get now
\begin{eqnarray*}
\lefteqn{{\mathbb{P}}^{\GG}(X_k^e-\ell^e_0=\ell_k^e,\, e\in E(\GG), 
A_{v,k}, ~\mbox{the walk does not leave}~ \GG_n~\mbox{by time k})}\\
&\le& \frac{\prod_{e\in E({\scriptsize \GG_n})}w(\ell_0^e)}{\min_{e \in E(\GG_n)}w(\ell_0^e)}{\sum_{i=1, v_i\in V(\GG_n)\atop \ell_k^{\{v,v_i\}}\ge 1}^{d_v}}\,\,\frac{w(\ell_k^{\{v,v_i\}}-1+\ell_0^e) +\sum_{e\in\GG_n\cap \NNs_{v_i}\atop e\neq \{v,v_i\}}w(\ell_k^e+\ell_0^e)}
{w(\ell_k^{\{v,v_i\}}-1+\ell_0^e)\prod_{e\in E({\scriptsize \GG_n})\setminus \{v,v_i\}} w(\ell_k^e+\ell_0^e)}\\
&&\times\frac{w(\ell_k^{\{v,v_i\}}-1+\ell_0^e)}{w(\ell_k^{\{v,v_i\}}-1+\ell_0^e) +\sum_{e\in\GG_n\cap \NNs_{v_i}\atop e\neq \{v,v_i\}}w(\ell_k^e+\ell_0^e)+\sum_{e\in\partial\GG_n\cap \NNs_{v_i}}w(\ell_0^e)}\\
&\le& \frac{\prod_{e\in E({\scriptsize \GG_n})}w(\ell_0^e)}{\min_{e \in E(\GG_n)}\, w(\ell_0^e) }\,\,\cdot\frac{\sum_{e\in \NNs_v,  e\in\GG_n}w(\ell_k^e+\ell_0^e)}
{\prod_{e\in E({\scriptsize \GG_n})} w(\ell_k^e+\ell_0^e)},
\end{eqnarray*}
which shows (\ref{constrgn}). In view of (\ref{constrgn}),  (\ref{condgn1a}) follows now immediately for all $v\in V(\GG_n)$.

\qed

\begin{remark}
\label{eachedgeinf}
\begin{itemize}
\item [(a)] Let $\GG$ be an infinite graph of bounded degree and assume that each edge $e\in E(\GG)$ has its own reinforcement function $w_e$. Provided that $w_e, e\in E(\GG),$ satisfy (\ref{Azero}) and (\ref{initialweight}), and that they also satisfy, similarly to (\ref{Eq: Stuck_new_edge_v}),  the uniform bound
\begin{equation}
\label{notstuck}
\inf_{\{v,v'\}}\prod_{i=1}^\infty\frac{\left(w_{\{v,v'\}}(i+\ell_0^{\{v,v'\}})\right)^2}{\big(w_{\{v,v'\}}(i+\ell_0^{\{v,v'\}})+\sum_{e\in {\cal{N}}(v)\atop e\neq {\{v,v'\}}}w_e(\ell_0^e)\big)\big({w_{\{v,v'\}}(i+\ell_0^{\{v,v'\}})+\sum_{e\in {\cal{N}}(v')\atop e\neq {\{v,v'\}}}w_{e}(\ell_0^e)}\big)}>0,
\end{equation}
we can easily show by Remark \ref{difw} and by similar arguments as in Theorem \ref{maininfedge} above, that 
\begin{equation}
\label{alwaysstuck}
{\mathbb{P}}^{\GG}(\GG_\infty~\mbox{has only one edge}) = 1.
\end{equation}
Sufficient conditions for (\ref{notstuck}) to hold are
\begin{equation}
\label{difedgweight}
\sup_{e\in E(\GG)}\sum_{i=1}^\infty\frac{1}{w_{e}(i+\ell_0^e)}<\infty~~~\mbox{and}~~~\sup_{e\in E(\GG)}w_e(\ell_0^e)<\infty.
\end{equation}
This result should be contrasted with Theorem 1.5 of Collevecchio, Cotar and LiCalzi \cite{CCL}(see also section 4.1 therein), where under the assumption $\sum_{i=1}^\infty\frac{1}{w_{e}(i+\ell_0^e)}<\infty,e\in E(\GG)$, a new phase transition was shown to exist for the related model of a generalized Polya urn model with infinitely many urns. More precisely, it was  proved therein that for the particular case where each urn $i, i\in\mathbb{N},$ follows the reinforcement function $w_i(n)=e^{i^3+n}$, every urn is visited only finitely many times. However, these particular reinforcement functions satisfy neither $\sup_{e\in E(\GG)}w_e(\ell_0^e)<\infty$ (not even when $\ell_0^e\equiv\ell_0,e\in E(\GG)$) nor (\ref{notstuck}). 
\item [(b)] Similarly to the reasoning in (a) above, one can show  for the reinforced walk from Remark \ref{difw} (b) that ${\mathbb{P}}^{\GG}(\GG_\infty~\mbox{has only one edge}) = 1$ on an infinite graph of bounded degree, under assumptions (\ref{Azero}), (\ref{initialweight}), (\ref{randweight}) (plus the equivalent to (\ref{notstuck})) on the weight functions $w_i$. 
\end{itemize}
\end{remark}


\section{Strongly vertex-reinforced random walks on general graphs}

In this section, we prove Theorem \ref{maininfvert}. As for the edge-reinforced walk, we will first work out the strategy for finite graphs in Subsection \ref{SFGraphvert} and then we will use the finite graph computations to extend our arguments in Subsection \ref{InfGraphvert} to the case of infinite connected graphs with bounded degree.

\subsection{Analysis on finite graphs}
\label{SFGraphvert}

Let $\GG$ be a finite graph, and abbreviate $\nn=|V(\GG)|$.
Moreover, denote the vertices of $\GG$ by $V(\GG)=\{v_1,v_2,\ldots v_\nn\}.$
If $v$ is an arbitrary vertex of the graph, we define as before $d_v:=\mbox{degree(v)}$. However, this time
$\NN_v:=\{v_1^v, v_2^v,\ldots v_{d_v}^v\}$ is the set of  vertices adjacent to $v$, as opposed to the set of edges incident to $v$ as was the case in Section 2. 
Recall that  $X_k^v$ equals the initial vertex weight
$\ell_0^v$ incremented by the number of times vertex $v$ has been visited by time $k$. 


\subsubsection{Bounds for the probabilities of the vertex weights order statistics}

Recall that $t_0:=0$.
Fix the initial position $I_{t_0}$ at some arbitrary vertex $v_0$. 
We re-label $X_k^v-\ell_0^v$ (the number of vertex visits at time $k \ge 0$) in increasing order. More precisely, we define the \textit{order statistics at time $k$} as a vector $R_k = (R_k^1,...,R_k^{\nn})$.  The components of this vector are the values of 
$$v \mapsto  X_k^v - \ell_0^v$$
put in non-increasing order; this defines the vector $R_k$ uniquely. Therefore, for all ${k}\ge 0$ we have 
\begin{equation}
\label{vertexorder}
0\le R_k^\nn\le R_k^{\nn-1}\le\ldots\le R_k^{1},~~~R_k^i\le R_{k+1}^i~\mbox{for all}~i=1,2,\ldots,\nn,~~\mbox{and}~~\sum_{j=1}^{\nn} R_k^j=k.
\end{equation}
Let $v$ be any vertex in $V(\GG)$. Then the total number of arrivals to $v$ by time $k$ differs by at most 1 from the total number of departures from $v$ by time $k$.  The initial vertex $v_0$ and the vertex at time $k$, if different from $v_0$, are the only vertices where the number of arrivals is not equal to the number of departures.
Coupled with (\ref{vertexorder}), this gives in particular that $2R_k^1\le k+1$, which in turn implies $(\nn-1) R_k^2\ge \sum_{i=2}^{\nn} R_k^i=k-R_k^1\ge( k-1)/2$. Thus, for all $\nn\ge 3$ we have
\begin{equation}
\label{impineq}
[{k}/\nn]\le R_k^1\le [{(k+1)}/2]~~\mbox{and}~~ [(k-1)/(2(\nn-1)]\le R_k^2\le [{k}/2].
\end{equation}
Then the following proposition gives \textit{path-independent} upper bounds on the distribution of the number of vertex traversals at time $k$.
\begin{proposition}
\label{Pfingraph1v}
Let $k\geq 1$ and
$v'\in V(\GG)$ and denote by $A_{v',k}$ the event $\{I_k=v'\}$.
Then for any $\ell_k^v \in\mathbb{N}$, $v\in V(\GG)$, and
$\sum_{v \in V({\scriptsize \GG})}\ell_k^v={k}$, we have 
\begin{equation}
\label{EPfingraph1v}
{\mathbb{P}}^{\GG}(X_k^v-\ell_0^v=\ell_k^v, v \in V(\GG), A_{v',k})
\leq 
\frac{\prod_{v\in V({\scriptsize \GG}), v\neq {v_0}} w(\ell_0^v)}
{\min_{v \in {\NN}_{v_0}}\, w(\ell_0^v) }
\cdot 
\frac{\sum_{v\in \NNs_{v'}}w(\ell_k^v+\ell_0^v)}
{\prod_{v\in V({\scriptsize \GG}), v\neq v'} w(\ell_k^v+\ell_0^v)}.
\end{equation}
\end{proposition}
\begin{proof}

As in Proposition 19 in \cite{CotarLimic} and as in Step 3 from the proof of Theorem \ref{maininfedge}, we will use induction on 
$k\ge 1$ to prove the above inequality. 
We note to begin with that the corresponding inequality to (\ref{EPfingraph1v}) for $k=0$ 
\begin{equation}
\label{level0vrrw}
{\mathbb{P}}^{\GG}(X_0^v=\ell_0^v, v \in V(\GG), A_{v',0})
\leq 
\frac{\sum_{v\in \NNs_{v'}}w(\ell_0^v)}
{\min_{v \in {\NN}_{v_0}}\, w(\ell_0^v) },
\end{equation}
clearly holds, since when the left-hand-side is $0$
the right-hand-side is positive, and when
the left-hand-side is $1$ (i.e., when $v'=v_0$) the right-hand-side is greater than $1$.

Now take $k\ge 1$ and consider the event on the left-hand-side.
For each $i=1,2,\ldots,n_{v'}$, let $v_i\in V(\GG)$ be the neighbour of 
$v'$. In order for the event 
$\{X_k^v-\ell_0^v=\ell_k^v, v \in V(\GG), A_{v',k}\}$
to happen, it must be 
$I_{k-1}=v_i$ for some
$v_i\sim v'$, $\ell_k^{v'}> \ell_0^{v'}$, and 
furthermore it must be $\{I_k=v'\}$. Therefore
\begin{multline*}
{\mathbb{P}}^{\GG}(X_k^v-\ell_0^v=\ell_k^v, v \in V(\GG), A_{v',k})\\
=\sum_{i=1}^{d_{v'}}{\mathbb{P}}^{\GG}(X_{k-1}^v=\ell_k^v+\ell_0^v,\, \forall v\neq v',
X_{k-1}^{v'}=\ell_k^{v'}-1+\ell_0^{v'}, A_{v_i,k-1})\\
\times\frac{w(\ell_k^{v'}-1+\ell_0^{v'})}{w(\ell_k^{v'}-1+\ell_0^{v'}) +\sum_{v\neq v'\in \NNs_{v_i}}w(\ell_k^v+\ell_0^v)}.
\end{multline*}
Given the above and (\ref{level0vrrw}), the proof follows now by the same induction arguments on $k\ge 1$ as in the proof of Proposition 19 from \cite{CotarLimic} and will be omitted.
\end{proof}

Let $\ell_0^{v_i}\in {\mathbb{N}}^+, v_i\in V(\GG), i=1,\ldots,\nn,$ be the initial vertex weights. Denote again by $S(\nn)$ the set of all permutations of $\{1,2,\ldots,\nn\}$ and by $\sigma$ an arbitrary element in $S(\nn)$. From Proposition \ref{Pfingraph1v}, it follows by the same reasoning as in Proposition \ref{edgegen} that
\begin{proposition}
\label{vertgen}
Let $k\geq 1$ and
$v'\in V(\GG)$ and denote by $A_{v',k}$ 
the event $\{I_k=v'\}$.
Then for any $\ell_k^i\in\mathbb{N}, i=1,\ldots,\nn,$ such that 
$0\le\ell_k^{\nn}\le\ldots \le\ell_k^i\le \ldots\ell_k^1\le {k}$ and
$\sum_{i=1}^{\nn}\ell_k^i={k}$, we have 
\begin{eqnarray}
\label{EPfingraph2v}
{\mathbb{P}}^{\GG}(R_k^i=\ell_k^i, i=1,\ldots, \nn, A_{v',k})
&\leq&\frac{\prod_{v\in V(\GG), v\neq v_0} w(\ell_0^v)}
{\min_{v \in V(\GG)}\, w(\ell_0^v)}
\cdot 
  \sum_{\sigma\in S(\nn)} \sum_{j=1}^{\nn}\frac{\sum_{i=1, i\neq j}^{\nn}w\left(\ell_k^{i}+\ell_0^{v_{\sigma(i)}}\right)}
{\prod_{i=1, i\neq j}^{\nn} w\left(\ell_k^i+\ell_0^{v_{\sigma(i)}}\right)}\nonumber\\
&=:&\Upsilon(w,\nn, (\ell_0^{v_i})_{i=1,\ldots,\nn}, (\ell_k^i)_{i=1,\ldots,\nn}),
\end{eqnarray}
and 
\begin{equation}
\label{EPfingraph3v}
{\mathbb{P}}^{\GG}(R_k^i=\ell_k^i, i=1,\ldots, \nn)
\leq \nn \Upsilon(w,\nn, (\ell_0^{v_i})_{i=1,\ldots,\nn}, (\ell_k^i)_{i=1,\ldots,\nn}).
\end{equation}
\end{proposition}
\begin{proof}
We will only prove (\ref{EPfingraph2v}) as (\ref{EPfingraph3v}) follows immediately from (\ref{EPfingraph2v}) by summing over all possible vertices $A_{v',k}, v'\in V(\GG)$. 

We note first that the event $\{R_k^i=\ell_k^i, i=1,\ldots, \nn\}$ is a union over $\sigma\in S(\nn)$ of the events $\{X_k^{v_{i}}-\ell_0^{v_{i}}=\ell_k^{\sigma(i)}, i=1,\ldots,\nn\}$.
Assume $v'=v_m$ and $\NN_{v'}=\{v_{j_1},\ldots, v_{j_{d_{v'}}}\}, m,j_1,\ldots, j_{d_{v'}}\in \{1,\ldots,\nn\}$, with $m, j_1, \ldots, j_{d_{v'}}$ all distinct, that is, $m\neq j_1\neq\ldots\neq j_{d_{v'}}$. Then
\begin{eqnarray*}
\lefteqn{{\mathbb{P}}^{\GG}(R_k^i=\ell_k^i, i=1,\ldots, \nn, A_{v',k})}\\
&\le&\sum_{\sigma\in S(\nn)} {\mathbb{P}}^{\GG}(X_k^{v_i}-\ell_0^{v_i}=\ell_k^{\sigma(i)}, i=1,\ldots,\nn, A_{v',k})\\
&\le&\sum_{\sigma\in S(\nn)}\,\,\frac{\prod_{v\in V(\GG), v\neq v_0} w(\ell_0^v)}
{\min_{v\in {\NN}_{v_0}}\, w(\ell_0^v)}
\cdot 
\frac{\sum_{s=1}^{d_{v'}}w(\ell_k^{\sigma(j_s)}+\ell_0^{v_{j_s}})}
{\prod_{i=1, i\neq m}^{\nn} w\left(\ell_k^{\sigma(i)}+\ell_0^{v_i}\right)}\\
&\le& \frac{\prod_{v\in V(\GG), v\neq v_0} w(\ell_0^v)}
{\min_{v\in {\NN}_{v_0}}\, w(\ell_0^v)}\cdot\sum_{\sigma\in S(\nn)}\frac{\sum_{i=1, i\neq m}^{\nn}w\left(\ell_k^{\sigma(i)}+\ell_0^{v_i}\right)}
{\prod_{i=1, i\neq m}^{\nn} w\left(\ell_k^{\sigma(i)}+\ell_0^{v_i}\right)}\\
&\le& \frac{\prod_{v\in V(\GG), v\neq v_0} w(\ell_0^v)}
{\min_{v\in {\NN}_{v_0}}\, w(\ell_0^v)}\cdot\sum_{\sigma\in S(\nn)}\sum_{j=1}^{\nn}\frac{\sum_{i=1, i\neq j}^{\nn}w\left(\ell_k^{i}+\ell_0^{v_{\sigma(i)}}\right)}
{\prod_{i=1, i\neq j}^{\nn} w\left(\ell_k^{i}+\ell_0^{v_{\sigma(i)}}\right)}.
\end{eqnarray*}

\end{proof}
\begin{remark}
\begin{itemize}
\item [(a)] While the order statistics formulas for edge-reinforced walk from Proposition \ref{Pfingraph1}  contain the weights corresponding to all the edges incident to any given vertex $v$, the similar formulas for vertex-reinforced walk from Proposition \ref{Pfingraph1v} contain the weights corresponding to all the vertices incident to any given vertex $v$, but crucially do not contain the weight for the vertex $v$. Formula-wise, there is one less term in each product than there was in the edge-reinforced random walk case. This greatly affects the results and the computations involved.
\item [(b)] Since in computing the various probabilities in Proposition \ref{Pfingraph1v} by means of Proposition \ref{Pfingraph1v} above we can re-scale $w(k)$ to $w(k)/w_{\max}(\ell_0),$ where $w_{\max}(\ell_0):=\max_{v\in V(\GG)} w(\ell_0^v)$, we will take in our computations below $w(\ell_0^v)\le 1$ for all $v\in V(\GG)$. This means that the term $\prod_{v\in V(\GG), v\neq v_0} w(\ell_0^v)$ from the formulas in Proposition \ref{vertgen} will not appear in our computations below.
\end{itemize}
\end{remark}

\subsubsection{Attraction sets on finite graphs}

We will first give an alternative definition of the event $\{\GG_\infty~\mbox{has exactly two vertices}\}$ by using the order statistics for the number of vertex traversals. First, since by definition the sequence of random variables $R_k^{3}, k\ge 1,$ is non-decreasing  in $k$ there exists $R_\infty^{3}=\lim_{k\rightarrow\infty} R_k^{3}$, which may or may not be finite. We have immediately for any finite graph that
$$\{\GG_\infty ~\mbox{has exactly two vertices}\}=\{R_\infty^3<\infty\}.$$
We will use our alternative definition above to prove in this section
\begin{theorem}
\label{mainvertfin}
Let $\GG$ be a finite graph with $\nn\ge 3$ vertices. If $w$ satisfies 
\begin{equation}
\label{maxfinv}
\max_{v\in V(\GG)}\sum _{i\geq 1}^{\infty}\frac{i}{w(i+\ell_0^v)} < \infty,
\end{equation}
then the walk traverses exactly two random neighbouring attracting vertices at all large times a.s.,
that is
$${\mathbb{P}}^{\GG}(\GG_\infty~\mbox{has exactly two vertices}) = 1.$$
\end{theorem}

In preparation for the proof of attraction sets, we will first obtain in Proposition \ref{2orderstatv} upper bounds for the distribution of $R_k^{3}$, that is, the distribution of the 3rd largest weight at time $k$.  For simplicity and clarity of computations, we will consider in detail in all our calculations below only the case where all vertices have \textit{equal} initial weights $\ell_0^v\equiv\ell_0\in\R^+$, the case with general initial weights $\ell_0^v\in\R^+,v\in V(\GG),$ following by similar arguments by means of Proposition \ref{vertgen}.  

\begin{proposition}
\label{2orderstatv}
Assume that (\ref{Azerov}) holds.  For $\ell_0\in\R^+$, let $\ell^v_0 = \ell_0$ for all vertices $v\in V(\GG)$. Then for all  $k\ge 1$ and $0\le\ell_k^3\le [k/3],\ell_k^3\in\mathbb{N}$, we have for $C(w,\nn, \ell_0)>0$ depending only on $w, \nn$ and $\ell_0$ 
\begin{equation*}
{\mathbb{P}}^{\GG}(R_k^{3}=\ell_k^{3})\le C(w,\nn, \ell_0)\left[\frac{\ell_k^3}{w(\ell_k^3+\ell_0)}+ \sum_{i=\ell_k^3}^{{k}-\ell_k^3}\frac{1}{w(i+\ell_0)}\right].
\end{equation*}



\end{proposition}
\begin{proof}
 Fix ${k}\ge 1$ and $\ell_k^{3}$ such that $0\le\ell_k^{3}\le [{k}/3]$. To simplify notation, we will denote by
 \begin{equation}
\label{lsumbig}
{\cal{L}}_{\nn,k}(\ell_k^3):=\!\!\!\{(\ell_k^1,\ell_k^2,\ell_k^4,\ldots, \ell_k^\nn)\in \mathbb{N}^{\nn-1}: \,0\le \ell_k^\nn\le\ldots\le\ell_k^3\le\ell_k^2\le\ell_k^1\le [(k+1)/2], \sum_{i=1\atop i\neq 3}^{\nn}\ell_k^i={k}-\ell_k^3\}.
\end{equation}
By (\ref{EPfingraph3v}) from Proposition \ref{vertgen}, we have
\begin{eqnarray}
\label{eqmainvert}
{\mathbb{P}}^{\GG}(R_k^{3}=\ell_k^{3})&\le&\sum_{(\ell_k^1,\ell_k^2,\ell_k^4,\ldots, \ell_k^\nn)\in {\cal{L}}_{\nn,k}(\ell_3)} {\mathbb{P}}^{\GG}(R_k^i=\ell_k^i, i=1,\ldots, \nn)\nonumber\\
& \leq&\frac{\nn\cdot\nn!}{w(\ell_0)}
\cdot 
\sum_{(\ell_k^1,\ell_k^2,\ell_k^4,\ldots, \ell_k^\nn)\in {\cal{L}}_{\nn,k}(\ell_3)}\,\,\,\,\sum_{j=1}^{\nn}
\frac{\sum_{i=1, i\neq j}^{\nn}w\left(\ell_k^{i}+\ell_0\right)}
{\prod_{i=1, i\neq j}^{\nn} w\left(\ell_k^i+\ell_0\right)},
\end{eqnarray}
We will next further estimate (\ref{eqmainvert}). We will first  look at the "warm-up" case $\nn=3$, which will give us insight into the computations for the more complicated $\nn\ge 4$ case.
\vspace{0.5mm}

(a) In this simpler case, (\ref{eqmainvert}) becomes
\begin{equation}
\label{easy1v}
{\mathbb{P}}^{\GG}(R_k^{3}=\ell_k^{3})\le\frac{6\cdot 3!}{w(\ell_0)}\sum_{(\ell_k^1,\ell_k^2)\in\mathbb{N}^2:0\le\ell_k^3\le\ell_k^2\le\ell_k^1\le [(k+1)/2]\atop\ell_k^1+\ell_k^2={k}-\ell_k^3}\left(\frac{1}{w(\ell_k^1+\ell_0)}+\frac{1}{w(\ell_k^2+\ell_0)}+\frac{1}{w(\ell_k^3+\ell_0)}\right).
\end{equation}
Since $\ell_k^1+\ell_k^2=k-\ell_k^3$, we have $2\ell_k^1\ge k-\ell_k^3$; using also that $\ell_k^3\le [k/3]$, we have $\ell_k^3\le [(k-\ell_k^3)/2]\le\ell_k^1\le \min([(k+1)/2], k-2\ell_k^3)$.
Then (\ref{easy1v}) becomes 
\begin{eqnarray*}
{\mathbb{P}}^{\GG}(R_k^{3}=\ell_k^{3})&\le&\frac{6\cdot 3!}{w(\ell_0)}\sum_{\ell_k^1=[(k-\ell_k^3)/2]}^{\min([(k+1)/2], k-2\ell_k^3)}\left(\frac{1}{w(\ell_k^3+\ell_0)}+\frac{1}{w(\ell_k^1+\ell_0)}+\frac{1}{w(k-\ell_k^3-\ell_k^1+\ell_0)}\right)\\
&\le&\frac{6\cdot 3!}{w(\ell_0)}\bigg[\sum_{\ell_k^1=[(k-\ell_k^3)/2]}^{[(k+1)/2]}\frac{1}{w(\ell_k^3+\ell_0)}+\sum_{\ell_k^1=\ell_k^3}^{k-2\ell_k^3}\bigg(\frac{1}{w(\ell_k^1+\ell_0)}+\frac{1}{w({k}-\ell_k^3-\ell_k^1+\ell_0)}\bigg)\bigg] \\
&\le&\frac{6\cdot 3!}{w(\ell_0)}\bigg[\frac{\ell_k^3}{w(\ell_k^3+\ell_0)}+2\sum_{i=\ell_k^3}^{{k}-2\ell_k^3}\frac{1}{w(i+\ell_0)}\bigg],
\end{eqnarray*}
where for the third inequality in the above we used in the first sum that the number of $\ell_k^1$ terms in the first sum is smaller than $\ell_k^3$; moreover, for the third sum in the inequality we made the substitution $i=k-\ell_k^3-\ell_k^1,$ with $\ell_k^3\le i\le k-2\ell_k^3$.

\vspace{0.5mm}

(b) In the more complicated $\nn\ge 4$ case, (\ref{eqmainvert}) becomes
\begin{eqnarray}
\label{eqmainverta}
{\mathbb{P}}^{\GG}(R_k^{3}=\ell_k^{3})& \leq&\frac{(\nn-1)\cdot\nn\cdot\nn!}{w(\ell_0)}
\cdot 
\sum_{(\ell_k^1,\ell_k^2,\ell_k^4,\ldots, \ell_k^\nn)\in {\cal{L}}_{\nn,k}(\ell_k^3)}\,\,\,\sum_{j,h=1\atop j\neq h}^{\nn}
\frac{1}
{\prod_{i=1, i\neq j, h}^{\nn} w\left(\ell_k^i+\ell_0\right)}\nonumber\\
&\le&\frac{(\nn-1)\cdot\nn\cdot\nn!}{w(\ell_0)}
\cdot 
\sum_{(\ell_k^1,\ell_k^2,\ell_k^4,\ldots, \ell_k^\nn)\in  {\cal{L}}_{\nn,k}(\ell_k^3)}\bigg[\frac{1}
{\prod_{i=3}^{\nn} w\left(\ell_k^i+\ell_0\right)}\nonumber\\
&&+\frac{1}{w(\ell_k^3+\ell_0)w(\ell_k^1+\ell_0)}\sum_{j=4}^{\nn}
\frac{1}
{\prod_{i=4, i\neq j}^{\nn} w\left(\ell_k^i+\ell_0\right)}\nonumber\\
&&+\frac{1}{w(\ell_k^3+\ell_0)w(\ell_k^2+\ell_0)}\sum_{j=4}^{\nn}
\frac{1}
{\prod_{i=4, i\neq j}^{\nn} w\left(\ell_k^i+\ell_0\right)}\nonumber\\
&&+\frac{1}{w(\ell_k^3+\ell_0)w(\ell_k^2+\ell_0)w(\ell_k^1+\ell_0)}\sum_{j,h=4\atop j\neq h}^{\nn}
\frac{1}
{\prod_{i=4, i\neq j, h}^{\nn} w\left(\ell_k^i+\ell_0\right)}\nonumber\\
&&+\frac{1}{w(\ell_k^1+\ell_0)w(\ell_k^2+\ell_0)}\sum_{j=4}^{\nn}
\frac{1}
{\prod_{i=4, i\neq j}^{\nn} w\left(\ell_k^i+\ell_0\right)}\bigg].
\end{eqnarray}

For the second inequality in the above, we split the terms in the double sum of the first inequality in five new sums, depending on which terms $w(\ell_k^i+\ell_0)$ and $w(\ell_k^j+\ell_0), i\neq j,$ are missing. We will next estimate separately each of the five sum terms in (\ref{eqmainverta}). Note also that the fourth sum term in the above only appears for $\nn\ge 5$, since for $\nn=3,4,$ the sums contain at most 2 product terms. 

We will use in our computations below the definition of $Q_m(a;,b;c), m, a, c\in\mathbb{N}, m\ge 1, b\in {\mathbb{R}}^+,$ from (\ref{eqndefq}) and formula (\ref{rec2}) from Proposition \ref{2orderstat}  (a). Moreover, to simplify formulas, we will denote in the proofs by (the possibly empty set)
\begin{multline}
\label{simpk}
{\cal{A}}_{\nn,k}(s,\ell_k^3):=\{(\ell_k^1,\ell_k^2)\in {\mathbb{N}}^2: \ell_k^1+\ell_k^2=s, \ell_k^3\le\ell_k^2\le\ell_k^1,\\ [{(k-1)}/2(\nn-1)]\le\ell_k^2\le\ell_k^1, [{k}/\nn]\le\ell_k^1\le [(k+1)/2]\},
\end{multline}
with the convention that summing over an empty set ${\cal{A}}_{\nn,k}(s,\ell_k^3)$ is equal to 0.

\textbf{Step 1:}  We will estimate here the first sum in (\ref{eqmainverta}), in which $w(\ell_k^1+\ell_0)$ and $w(\ell_k^2+\ell_0)$ are missing.

For any $(\ell_k^1,\ell_k^2,\ell_k^4,\ldots, \ell_k^\nn)\in {\cal{L}}_{\nn,k}(\ell_k^3)$, we have $2\ell_k^1+(\nn-3)\ell_k^3\ge\sum_{i=1,i\neq 3}\ell_k^i= k-\ell_k^3$, from which $2\ell_k^1\ge \max(k- (\nn-2)\ell_k^3, 2\ell_k^3)$. Therefore, if $k> \nn\ell_k^3$ we have
$$\ell_k^3\le [(k-(\nn-2)\ell_k^3)/2]\le \ell_k^1\le [(k+1)/2].$$
If $k\le \nn\ell_k^3$, and in view of $\ell_k^1\le [(k+1)/2]$, we have $\ell_k^3\le\ell_k^1\le 1+[\nn\ell_k^3/2]$. 
Putting both cases together, we get that for each fixed $\ell_k^1+\ell_k^2$, the number of  $(\ell_k^1,\ell_k^2)$ pairs in the first sum in (\ref{eqmainverta}) is smaller than $\nn\ell_k^3$. Therefore, since $2\ell_k^3\le\ell_k^1+\ell_k^2\le k-\ell_k^3$, the first term in (\ref{eqmainverta}) becomes
\begin{eqnarray}
\label{sum1va}
\lefteqn{\sum_{(\ell_k^1,\ell_k^2,\ell_k^4,\ldots, \ell_k^\nn)\in  {\cal{L}}_{\nn,k}(\ell_k^3)}\frac{1}
{\prod_{i=3}^{\nn} w\left(\ell_k^i+\ell_0\right)}}\nonumber\\
&\le&\frac{1}{w(\ell_k^3+\ell_0)}\sum_{s=2\ell_k^3}^{{k}-\ell_k^3}\,\,\sum_{(\ell_k^1,\ell_k^2)\in {\mathbb{N}}^2,\ell_k^1+\ell_k^2=s\atop |(\ell_k^1,\ell_k^2)|\le\nn\ell_k^3}\,\,\sum_{(\ell_k^4,\ldots, \ell_k^\nn)\in {\mathbb{N}}^{\nn-3}: \,0\le \ell_k^\nn\le\ldots\le\ell_k^3\atop \sum_{i=4}^{\nn}\ell_k^i=k-\ell_k^3-s}\frac{1}
{\prod_{i=4}^{\nn} w\left(\ell_k^i+\ell_0\right)}\nonumber\\
&\le&\frac{\nn\ell_k^3}{w(\ell_k^3+\ell_0)}\sum_{s=2\ell_k^3}^{{k}-\ell_k^3}\,\,\,\,\sum_{(\ell_k^4,\ldots, \ell_k^\nn)\in {\mathbb{N}}^{\nn-3}: \,0\le \ell_k^\nn\le\ldots\le\ell_k^3\atop \sum_{i=4}^{\nn}\ell_k^i=k-\ell_k^3-s}\frac{1}
{\prod_{i=4}^{\nn} w\left(\ell_k^i+\ell_0\right)}\nonumber\\
&=&\frac{\nn\ell_k^3}{w(\ell_k^3+\ell_0)}\sum_{s=2\ell_k^3}^{{k}-\ell_k^3}\,\,Q_{\nn-3}(k-\ell_k^3-s;\ell_0;\ell_k^3)\le\frac{\nn\ell_k^3}{w(\ell_k^3+\ell_0)} (c(\ell_0))^{\nn-3},
\end{eqnarray}
where for all $2\ell_k^3\le s\le k-\ell_k^3$, we have $k-\ell_k^3-s\ge 0$ since $0\le \ell_k^3\le [k/3]$. Moreover, for the last inequality in (\ref{sum1va}) we used (\ref{rec2}), and where we recall that $c(\ell_0):=\sum_{i=0}^\infty 1/w(\ell_0+i)$.

\textbf{Step 2:} We will estimate here the second, the third and the fourth terms in (\ref{eqmainverta}). For the second term,  where in each of the sums $w(\ell_k^2+\ell_0)$ and one of $w(\ell_k^j+\ell_0),  4\le j\le\nn,$ are missing, we have
\begin{eqnarray}
\label{sum2vbach}
\lefteqn{\sum_{(\ell_k^1,\ell_k^2,\ell_k^4,\ldots, \ell_k^\nn)\in {\cal{L}}_{\nn,k}(\ell_k^3)}
\frac{1}{w(\ell_k^3+\ell_0)w(\ell_k^1+\ell_0)}\sum_{j=4}^{\nn}\frac{1}{\prod_{i=4, i\neq j}^{\nn} w\left(\ell_k^i+\ell_0\right)}}\nonumber\\
&\le&\frac{1}{w(\ell_k^3+\ell_0)}\sum_{s=2\ell_k^3}^{{k}-\ell_k^3}\,\,\sum_{(\ell_k^1,\ell_k^2)\in {\cal{A}}_{\nn,k}(s,\ell_k^3)}\,\,\frac{1}{w(\ell_k^1+\ell_0)}\,\sum_{(\ell_k^4,\ldots, \ell_k^\nn)\in {\mathbb{N}}^{\nn-3}: \,0\le \ell_k^\nn\le\ldots\le\ell_k^3\atop \sum_{i=4}^{\nn}\ell_k^i={k}-\ell_k^3-s}\,\,\,\sum_{j=4}^{\nn}\frac{1}{\prod_{i=4, i\neq j}^{\nn} w\left(\ell_k^i+\ell_0\right)}\nonumber\\
&\le&\frac{1}{w(\ell_k^3+\ell_0)}\bigg(\sum_{\ell_k^1=\max([{k}/\nn],\ell_k^3)}^{{k}-2\ell_k^3}\frac{1}{w(\ell_k^1+\ell_0)}\bigg)\sum_{s=2\ell_k^3}^{{k}-\ell_k^3}\,\,\bigg[\sum_{(\ell_k^4,\ldots, \ell_k^\nn)\in {\mathbb{N}}^{\nn-3}: \,0\le \ell_k^\nn\le\ldots\le\ell_k^3\atop \sum_{i=4}^{\nn}\ell_k^i={k}-\ell_k^3-s}\,\,\,\bigg(\frac{1}{\prod_{i=5}^{\nn} w\left(\ell_k^i+\ell_0\right)}\nonumber\\
&&+\sum_{j=5}^{\nn}\frac{1}{\prod_{i=4}^{j-1}w(\ell_k^i+\ell_0)\times \prod_{i=j+1}^{\nn}w(\ell_k^{i}+\ell_0)}\bigg)\bigg].
\end{eqnarray}
In the above, we summed for the first inequality first after all the values that $(\ell_k^1,\ell_k^2)$ can take, and then after all the values that the remaining $(\ell_k^4,\ldots,\ell_k^{\nn})$ can take, given $(\ell_k^1,\ell_k^2)$. For the second inequality, we used that for each fixed $s$, with $2\ell_k^3\le s\le{k}-\ell_k^3$,
 $$\sum_{(\ell_k^1,\ell_k^2)\in {\cal{A}}_{\nn,k}(s,\ell_k^3)}\,\,\frac{1}{w(\ell_k^1+\ell_0)}\le \sum_{\ell_k^1=\max([{k}/\nn],\ell_k^3)}^{{k}-2\ell_k^3}\frac{1}{w(\ell_k^1+\ell_0)},$$
 and that $\sum_{\ell_k^1=\max([{k}/\nn],\ell_k^3)}^{{k}-2\ell_k^3}1/w(\ell_k^1+\ell_0)$ is independent of $s$. To get the big inner sum in the second inequality, we expanded the sum over $(\ell_k^4,\ldots, \ell_k^\nn)$ from the first inequality into two new sums, depending on which term $w(\ell_k^j+\ell_0), j\geq 4$, is missing from the product. 
 
We will expand next further the two inner sums in (\ref{sum2vbach}). For the first inner sum in (\ref{sum2vbach}), we have
\begin{eqnarray}
\label{sum2vbr1}
\sum_{(\ell_k^4,\ldots, \ell_k^\nn)\in {\mathbb{N}}^{\nn-3}: \,0\le \ell_k^\nn\le\ldots\le\ell_k^3\atop \sum_{i=4}^{\nn}\ell_k^i={k}-\ell_k^3-s}\,\,\,\frac{1}{\prod_{i=5}^{\nn} w\left(\ell_k^i+\ell_0\right)}
&\le&\sum_{\ell_k^4=0}^{\min(\ell_k^3, k-\ell_k^3-s)}\!\!\!\!\!\!\!\sum_{(\ell_k^5,\ldots, \ell_k^\nn)\in {\mathbb{N}}^{\nn-4}: 
\,0\le \ell_k^\nn\le\ldots\le\ell_k^4\atop \sum_{i=5}^{\nn}\ell_k^i={k}-\ell_k^3-s-\ell_k^4}\,\,\,\frac{1}{\prod_{i=5}^{\nn} w\left(\ell_k^i+\ell_0\right)}\nonumber\\
&\le&\sum_{\ell_k^4=0}^{\min(\ell_k^3, k-\ell_k^3-s)}Q_{\nn-4}({k}-s-\ell_k^3-\ell_k^4;\ell_0;\infty),
\end{eqnarray}
where for the first inequality we summed first after $\ell_k^4$ and then after the remaining $(\ell_k^5,\ldots, \ell_k^\nn)$ values, given $\ell_k^4$; we also used $\sum_{i=5}^{\nn}\ell_k^i={k}-\ell_k^3-s-\ell_k^4\ge 0$ and $0\le\ell_k^4\le\ell_k^3$, to obtain $0\le\ell_k^4\le\min (\ell_k^3, k-\ell_k^3-s)$. For the second inequality we used the definition of $Q_{\nn-4}$ from (\ref{eqndefq}).  

For the second inner sum in (\ref{sum2vbach}), we similarly have for each $5\le j\le\nn$
\begin{eqnarray}
\label{sum2vbr2}
\lefteqn{\sum_{(\ell_k^4,\ldots, \ell_k^\nn)\in {\mathbb{N}}^{\nn-3}: \,0\le \ell_k^\nn\le\ldots\le\ell_k^3\atop \sum_{i=4}^{\nn}\ell_k^i={k}-\ell_k^3-s}\,\,\,\frac{1}{\prod_{i=4}^{j-1}w(\ell_k^i+\ell_0)\times \prod_{i=j+1}^{\nn}w(\ell_k^{i}+\ell_0)}}\nonumber\\
&\le&\sum_{\ell_k^j=0}^{\min(\ell_k^3, k-\ell_k^3-s)}\,\,\,\sum_{(\ell_k^{4},\ldots, \ell_k^{j-1})\in {\mathbb{N}}^{j-4}: \,\ell_k^j\le\ell_k^{j-1}\ldots\le\ell_k^{4}\le\ell_k^3 \atop {(\ell_k^{j+1},\ldots, \ell_k^\nn)\in {\mathbb{N}}^{\nn-j}: \,0\le \ell_k^\nn\le\ldots\le\ell_k^{j}\atop \sum_{i=4, i\neq j}^{\nn}\ell_k^i={k}-\ell_k^3-s-\ell_k^j}} \frac{1}{\prod_{i=4}^{j-1}w(\ell_k^i+\ell_0)\times \prod_{i=j+1}^{\nn}w(\ell_k^{i}+\ell_0)}\nonumber\\
&\le&\sum_{\ell_k^j=0}^{\min(\ell_k^3, k-\ell_k^3-s)}   \,\,\,\sum_{(h^1,\ldots, h^{\nn-4})\in {\mathbb{N}}^{\nn-4}: 0\le h^{\nn-4}\le\ldots\le h^1\le\ell_k^3\atop \sum_{i=1}^{\nn-4}h^i={k}-\ell_k^3-s-\ell_k^j} \frac{1}{\prod_{i=1}^{\nn-4}w(h^i+\ell_0)}\nonumber\\
&\le&\sum_{\ell_k^j=0}^{\min(\ell_k^3, k-\ell_k^3-s)} Q_{\nn-4}({k}-\ell_k^3-s-\ell_k^j;\ell_0;\infty).
\end{eqnarray}
To get the first inequality in (\ref{sum2vbr2}), we summed first after $\ell_k^j$ and then after the remaining values.  For the second inequality we removed the restriction on $(\ell_k^{i-1},\ell_k^{i+1})$ that $\ell_k^{j+1}\le\ell_k^j\le\ell_k^{j-1}$, and for the last inequality we used the definition of $Q_{\nn-4}$ from (\ref{eqndefq}).  

Combining (\ref{sum2vbach}), (\ref{sum2vbr1}) and (\ref{sum2vbr2}), we get
\begin{eqnarray}
\label{sum2vb}
\lefteqn{\sum_{(\ell_k^1,\ell_k^2,\ell_k^4,\ldots, \ell_k^\nn)\in {\cal{L}}_{\nn,k}(\ell_k^3)}
\frac{1}{w(\ell_k^3+\ell_0)w(\ell_k^1+\ell_0)}\sum_{j=4}^{\nn}\frac{1}{\prod_{i=4, i\neq j}^{\nn} w\left(\ell_k^i+\ell_0\right)}}\nonumber\\
&\le&\frac{1}{w(\ell_k^3+\ell_0)}\bigg(\sum_{\ell_k^1=\max([{k}/\nn],\ell_k^3)}^{{k}-2\ell_k^3}\frac{1}{w(\ell_k^1+\ell_0)}\bigg)\sum_{s=2\ell_k^3}^{{k}-\ell_k^3}\,\,\bigg[ \sum_{\ell_k^4=0}^{\min(\ell_k^3, k-\ell_k^3-s)}Q_{\nn-4}({k}-s-\ell_k^3-\ell_k^4;\ell_0;\infty)\nonumber\\
&&+\sum_{j=5}^\nn \sum_{\ell_k^j=0}^{\min(\ell_k^3, k-\ell_k^3-s)} Q_{\nn-4}({k}-\ell_k^3-s-\ell_k^j;\ell_0;\infty)\bigg],
\end{eqnarray}
where for $\nn=4$, the sums over the $Q_{\nn-4}$ terms do not appear in the formula. Even though we could now easily bound (\ref{sum2vb}) under assumption (\ref{maxfinv}) by using (\ref{rec2}), we will next estimate the sums in (\ref{sum2vb}) under the weaker condition (\ref{Azerov}) in such a way that we can re-use the bounds later on in the proof of Lemma \ref{specialcase} (a).
For the first sum, we have for $0\le\ell_k^3\le [k/4]$ (which is equivalent to $2\ell_k^3\le k-2\ell_k^3$)
\begin{eqnarray}
\label{2vba}
\lefteqn{\sum_{s=2\ell_k^3}^{{k}-\ell_k^3}\,\,\sum_{\ell_k^4=0}^{\min(\ell_k^3, k-\ell_k^3-s)}Q_{\nn-4}({k}-s-\ell_k^3-\ell_k^4;\ell_0;\infty)}\nonumber\\
&\le&\sum_{s=2\ell_k^3}^{{k}-2\ell_k^3}\,\,\sum_{\ell_k^4=0}^{\ell_k^3}Q_{\nn-4}({k}-s-\ell_k^3-\ell_k^4;\ell_0;\infty)+\sum_{s=k-2\ell_k^3}^{k-\ell_k^3}\sum_{\ell_k^4=0}^{k-\ell_k^3-s}Q_{\nn-4}({k}-s-\ell_k^3-\ell_k^4;\ell_0;\infty)\nonumber\\
&\le&\sum_{\ell_k^4=0}^{\ell_k^3}\sum_{s=2\ell_k^3}^{{k}-2\ell_k^3}\,\,Q_{\nn-4}({k}-s-\ell_k^3-\ell_k^4;\ell_0;\infty)+\sum_{\ell_k^4=0}^{\ell_k^3}\,\,\,\sum_{s=k-2\ell_k^3}^{{k}-\ell_k^3-\ell_k^4}\,\,Q_{\nn-4}({k}-s-\ell_k^3-\ell_k^4;\ell_0;\infty)\nonumber\\
&\le& 2\sum_{\ell_k^4=0}^{\ell_k^3}(c(\ell_0))^{\nn-4}=2(c(\ell_0))^{\nn-4}(\ell_k^3+1),
\end{eqnarray}
where for the first equality in (\ref{2vba}) we used that 
\begin{eqnarray}
\label{minineq}
\min(\ell_k^3, k-\ell_k^3-s)&=&\bigg\{\begin{array}{cc}
\ell_k^3,&\mbox{if}~~~2\ell_k^3\le s\le k-2\ell_k^3\\
k-\ell_k^3-s,&~~~~~\mbox{if}~~~k-2\ell_k^3\le s\le k-\ell_k^3.
\end{array}
\end{eqnarray}
For the first inequality we changed the summation order between $s$ and $\ell_k^4$, and we used in the second double sum that $k-\ell_k^3-s\le\ell_k^3$, if $k-2\ell_k^3\le s\le k-\ell_k^3$; for the second inequality we used (\ref{rec2}). A similar inequality to (\ref{2vba}) holds for $[k/4]\le\ell_k^3\le [k/3]$, the main difference being that now only the sum over $k-2\ell_k^3\le s\le k-\ell_k^3$ is possible. An inequality of the same form holds also for the second sum in (\ref{sum2vb}).
From (\ref{sum2vbach}), (\ref{sum2vb}) and (\ref{2vba}), we get that
\begin{multline}
\label{2vbc}
\sum_{(\ell_k^1,\ell_k^2,\ell_k^4,\ldots, \ell_k^\nn)\in {\mathbb{N}}^{\nn-1}: \,0\le \ell_k^\nn\le\ldots\le\ell_k^3\le\ell_k^2\le\ell_k^1\atop \sum_{i=1, i\neq 3}^{\nn}\ell_k^i={k}-\ell_k^3}
\frac{1}{w(\ell_k^3+\ell_0)w(\ell_k^1+\ell_0)}\sum_{j=4}^{\nn}\frac{1}{\prod_{i=4, i\neq j}^{\nn} w\left(\ell_k^i+\ell_0\right)}\\
\le \bar{c}(\ell_0,\nn)\frac{\ell_k^3}{w(\ell_k^3+\ell_0)} \bigg(\sum_{\ell_k^1=\max([{k}/\nn],\ell_k^3)}^{{k}-2\ell_k^3}\frac{1}{w(\ell_k^1+\ell_0)}\bigg),
\end{multline}
for some  $\bar{c}(\ell_0,\nn)>0$ which is independent of $\ell_k^i,i=1,\ldots,\nn$. A similar argument and bound as in (\ref{2vbc}) holds also for the third and for the fourth sums in (\ref{eqmainverta}) and will be omitted.

\textbf{Step 3:} We will estimate here the fifth sum in (\ref{eqmainverta}), in which $w(\ell_k^3+\ell_0)$ is missing. Similarly to Step 2, we have
\begin{eqnarray}
\label{sum3vb}
\lefteqn{\sum_{(\ell_k^1,\ell_k^2,\ell_k^4,\ldots, \ell_k^\nn)\in {\cal{L}}_{\nn,k}(\ell_k^3)}\frac{1}{w(\ell_k^1+\ell_0)w(\ell_k^2+\ell_0)}\sum_{j=4}^{\nn}\frac{1}{\prod_{i=4, i\neq j}^{\nn} w\left(\ell_k^i+\ell_0\right)}}\nonumber\\
&\le&\sum_{s=2\ell_k^3}^{{k}-\ell_k^3}\,\,\sum_{(\ell_k^1,\ell_k^2)\in {\cal{A}}_{\nn,k}(s,\ell_k^3)}\,\,\,\frac{1}{w(\ell_k^1+\ell_0)w(\ell_k^2+\ell_0)}\,\sum_{(\ell_k^4,\ldots, \ell_k^\nn)\in {\mathbb{N}}^{\nn-3}: \,0\le \ell_k^\nn\le\ldots\le\ell_k^3\atop \sum_{i=4}^{\nn}\ell_k^i={k}-\ell_k^3-s}\,\,\,\sum_{j=4}^{\nn}\frac{1}{\prod_{i=4\atop i\neq j}^{\nn} w\left(\ell_k^i+\ell_0\right)}\nonumber\\
&\le&\sum_{s=2\ell_k^3}^{{k}-\ell_k^3}\,\, \,\,\sum_{(\ell_k^1,\ell_k^2)\in {\cal{A}}_{\nn,k}(s,\ell_k^3)}\,\,\,\frac{1}{w(\ell_k^1+\ell_0)w(\ell_k^2+\ell_0)} \,\,\bigg[\sum_{\ell_k^4=0}^{\min(\ell_k^3, k-\ell_k^3-s)}Q_{\nn-4}({k}-s-\ell_k^3-\ell_k^4;\ell_0;\infty)\nonumber\\
&&+\sum_{j=5}^\nn\,\,\,\,\sum_{\ell_k^j=0}^{\min(\ell_k^3, k-\ell_k^3-s)} Q_{\nn-4}( {k}-\ell_k^3-s-\ell_k^j;\ell_0;\infty) \bigg]\nonumber\\
&\le&(\nn-4)(c(\ell_0))^{\nn-4}\sum_{s=2\ell_k^3}^{{k}-\ell_k^3}\,\, \,\,\sum_{(\ell_k^1,\ell_k^2)\in {\cal{A}}_{\nn,k}(s,\ell_k^3)}\,\,\,\frac{1}{w(\ell_k^1+\ell_0)w(\ell_k^2+\ell_0)} ,
\end{eqnarray}
where for $\nn=4$, the sums over the $Q_{\nn-4}$ terms do not appear in the formula. We summed for the first inequality in (\ref{sum3vb}), first after all the values that $(\ell_k^1,\ell_k^2)$ can take, and then after all the values that the remaining $(\ell_k^4,\ldots,\ell_k^{\nn})$ can take, given $(\ell_k^1,\ell_k^2)$. For the second inequality, we used (\ref{sum2vbr1}) and (\ref{sum2vbr2}) to further bound the sum over the $(\ell_k^4,\ldots,\ell_k^{\nn})$ values, and for the last inequality we used (\ref{rec2}).

It remains to estimate the double sum in the last inequality in (\ref{sum3vb}). We have
\begin{eqnarray}
\label{3vba}
\lefteqn{\sum_{s=2\ell_k^3}^{{k}-\ell_k^3}\,\,\,\sum_{(\ell_k^1,\ell_k^2)\in {\cal{A}}_{\nn,k}(s,\ell_k^3)}\,\,\,\frac{1}{w(\ell_k^1+\ell_0)w(\ell_k^2+\ell_0)}}\nonumber\\
&\le&\sum_{\ell_k^1=\max([{k}/\nn],\ell_k^3)}^{{k}-2\ell_k^3}\,\sum_{s=\ell_k^1+(k-1)/2(\nn-1)}^{{k}-\ell_k^3}\frac{1}{w(\ell_k^1+\ell_0)w(s-\ell_k^1+\ell_0)},
\end{eqnarray}
where in the above we recalled (\ref{simpk}) and changed the summation order between $s$ and $\ell_k^1$, and we used that $s=\ell_k^1+\ell_k^2\ge \ell_k^1+{(k-1)}/2(\nn-1)$.
From (\ref{sum3vb}) and (\ref{3vba}), we therefore get for some $\tilde{c}(\ell_0,\nn)>0$ that
\begin{multline}
\label{3vbc}
\sum_{(\ell_k^1,\ell_k^2,\ell_k^4,\ldots, \ell_k^\nn)\in {\cal{L}}_{\nn,k}(\ell_k^3)}\frac{1}{w(\ell_k^1+\ell_0)w(\ell_k^2+\ell_0)}\sum_{j=4}^{\nn}\frac{1}{\prod_{i=4, i\neq j}^{\nn} w\left(\ell_k^i+\ell_0\right)}\\
\le \tilde{c}(\ell_0,\nn) \sum_{\ell_k^1=\max([{k}/\nn],\ell_k^3)}^{{k}-2\ell_k^3} \frac{1}{w(\ell_k^1+\ell_0)}\sum_{s=\ell_k^1+{(k-1)}/2(\nn-1)}^{{k}-\ell_k^3}\,\, \frac{1}{w(s-\ell_k^1+\ell_0)}.
\end{multline}
\vspace{1.5mm}
The statement of the theorem follows now by combining the estimates from Steps 1-3, and by bounding the inner sum in (\ref{3vbc}) by means of (\ref{Azerov}). 

\vspace{0.5mm}

\end{proof}

\noindent\textbf{Proof of Theorem \ref{mainvertfin}}

We will show next that $R_\infty^3<\infty$ a.s., which will imply the statement of the theorem. For simplicity and clarity of computations, we will again restrict ourselves to the case with $\ell_0^v\equiv\ell_0\in\R^+,$ $v\in V(\GG)$, the case with general initial weights following by similar arguments by means of (\ref{EPfingraph3v}) and of a generalization to Proposition \ref{2orderstatv}. Under assumption (\ref{maxfinv}), and taking $p_\ell:=(\ell+\ell_0) /w(\ell+\ell_0)$ in Lemma \ref{gexistence}, there exists $g\colon \bN \rightarrow \left[0, \infty\right)$ such that (i) $g\left(\cdot\right)$ is increasing, $\textstyle \lim_{\ell \uparrow \infty}g\left(\ell\right)=\infty$ and (ii) $\sum _{\ell= 1}^{\infty}\ell g\left(\ell+\ell_0\right)/w(\ell+\ell_0) < \infty$. Since g is an increasing function, we have by the monotone convergence theorem that
\begin{equation}
\label{ineq1v}
{\mathbb{E}}^{\GG}(g(R_\infty^3))=\lim_{k\rightarrow\infty}{\mathbb{E}}^{\GG}(g(R_k^3))=\lim_{k\rightarrow\infty}\sum_{\ell_k^3=0}^{[{k}/3]} g(\ell_k^3){\mathbb{P}}^{\GG}(R_k^3=\ell_k^3).
\end{equation}
Then in view of Proposition \ref{2orderstatv}, (\ref{ineq1v}) becomes
\begin{eqnarray}
\label{ineq1va}
{\mathbb{E}}^{\GG}(g(R_\infty^3))&\le& C(w,\nn, \ell_0)\cdot\lim_{k\rightarrow\infty}\sum_{\ell_k^3=0}^{[{k}/3]}g(\ell_k^3)\bigg[\frac{\ell_k^3}{w(\ell_k^3+\ell_0)}+\sum_{i=\ell_k^3}^{{k}-\ell_k^3}\frac{1}{w(i+\ell_0)}\bigg].
\end{eqnarray}
By the hypothesis on $g$ and $w$, we have
$$\sum_{\ell_k^3=0}^{\infty}\frac{\ell_k^3\,g(\ell_k^3+\ell_0)}{w(\ell_k^3+\ell_0)}<M.$$
 It remains to bound the double sum term in (\ref{ineq1va}). In view of $\ell_k^3\le i\le{k}-\ell_k^3$, we have for $0\le\ell_k^3\le [{k}/3]$
\begin{eqnarray}
\label{3rdtermv}
\sum_{\ell_k^3=0}^{[{k}/3]}g(\ell_k^3)\sum_{i=\ell_k^3}^{{k}-\ell_k^3}\frac{1}{w(i+\ell_0)}&\le&\sum_{i=0}^{{k}}\sum_{\ell_k^3=0}^{i}\frac{g(\ell_k^3)}{w(i+\ell_0)}\le \sum_{i=0}^{{k}}\sum_{\ell_k^3=0}^{i}\frac{g(i+\ell_0)}{w(i+\ell_0)}\nonumber\\
&\le&\sum_{i=0}^{{k}}\frac{(i+1)g(i+\ell_0)}{w(i+\ell_0)}<2M+g(\ell_0)/w(\ell_0),
\end{eqnarray}
where we used in the above the monotonicity of $g$ and the hypothesis assumption on $w$. From (\ref{ineq1va}) and (\ref{3rdtermv}), it follows that ${\mathbb{E}}^{\GG}(g(R_\infty^3))<\infty$ which implies that a.s. $g(R_\infty^3)<\infty$. Combining this with the fact that $g$ is increasing and $\lim_{\ell\uparrow\infty} g(\ell)=\infty$, we get that a.s. $R_\infty^3<\infty$.
\qed

\begin{remark}
\label{problem}
\begin{itemize}
\item [(a)] Just as for ERRW, we can easily extend the result in Theorem \ref{mainvertfin}, by the same reasoning as in Theorem \ref{mainvertfin}, to the case where each vertex $v\in V(\GG)$ has its own weight function $w_v$ satisfying
$$\max_{v\in V(\GG)}\sum_{\ell=1}^\infty\frac{\ell}{w_v(\ell+\ell_0^v)}<\infty,\forall v\in V(\GG).$$
\item [(b)] We can also apply our method to the case when we add one loop to
each site, i.e., when at each step, independently of the actual position
of the walk, the probability to jump to some site $i$ is proportional to $w(X_n(i))$. Then the order statistics bounds are similar to the ones for ERRW (i.e., there is one more term in each product compared to the bounds for VRRW in Proposition \ref{vertgen}). This implies that the walk gets attracted to 2 vertices a.s under (\ref{Azerov}).
\item [(c)] Our method currently breaks down when $w$ does not satisfy (\ref{maxfinv}), but it still satisfies (\ref{Azerov}). As explained in Remark \ref{conjecture} c), we do not expect Theorem \ref{mainvertfin} to hold then for all finite/infinite graphs with vertex-reinforcement satisfying (\ref{Azerov}). However, we are also at present unable to prove the conjecture stated in Remark \ref{conjecture} c) as we cannot compute tight upper bounds for the distribution of $R_k^j, j\ge 4,$ when (\ref{maxfinv}) does not hold. The reason is that, unless we take the geometry of the graph into account, the corresponding bound for $R_k^j, j\ge 4,$ from Step 1 in Proposition \ref{2orderstatv} would still contain a term depending on $\ell_k^3$. This implies in particular that, if (\ref{maxfinv}) does not hold, our bound from Step 1 would blow up when $k$ grows large. 
\end{itemize} 
\end{remark}
In the case of bipartite graphs, where we can decouple nearest-neighbour vertices, we can bypass the issues arising in Step 1 of Proposition \ref{2orderstatv} for super-linear $w$ not satisfying  (\ref{maxfinv}), and we can show
\begin{Lemma}
\label{specialcase}
Let $\Lambda\subset{\mathbb {Z}}^d$ be a finite connected subset in ${\mathbb{Z}}^d$, or more generally any finite bipartite graph, with $|\Lambda|=\nn\ge 4$. If $w$ satisfies either: 
\begin{itemize}
\item [(a)] \begin{equation}
\label{supinfv1'}
\max_{v\in V(\GG)}\sum _{i\geq 1}^{\infty}\frac{i^{1/2}}{w(i+\ell_0^v)} < \infty,~~\mbox{or}
\end{equation}
\item [(b)] $w$ satisfies (\ref{Azerov}) and 
\begin{equation}
\label{bipbip}
\max_{v\in V(\GG)} \sup_{i\ge 1}\frac{i}{w(i+\ell_0^v)}<\infty,
\end{equation}
\end{itemize}
then the walk traverses exactly two random neighbouring attracting vertices at all large times a.s..
\end{Lemma}
\begin{proof}
We will start by obtaining upper bounds on ${\mathbb{P}}^{\GG}(R_k^{3}=\ell_k^{3}), 0\le\ell_k^3\le [k/3]$, as in Proposition \ref{2orderstatv} above; thus, we need to bound the five terms in (\ref{eqmainverta}). 
With the upper bounds in place, the proof will then proceed by the same arguments as in the proof of Theorem \ref{mainvertfin}. To get the upper bounds, we will first need to work out the possible values of $R_k$ in the special case of bipartite graphs. We note here that we have a shorter proof for (a) only, based on the key observation that in bipartite graphs, for each fixed $R_k^1+R_k^2$ the number of  $(R_k^1,R_k^2)$ pairs is less or equal than $1+(\nn-3)R_k^4$. (See the proof of Lemma \ref{triangle-free case} below for this argument in the more general case of triangle-free graphs). However, we have chosen to provide below a unifying proof for both (a) and (b). 

As in the proof of Proposition \ref{2orderstatv}, we will consider in detail in all our calculations below only the case when all vertices have \textit{equal} initial weights $\ell_0^v\equiv\ell_0\in\R^+$.

\noindent  
Recall that the total number of arrivals to any vertex $v\in V(\Lambda)$ by time $k$ is the same as the total number of departures from $v$ by time $k$, the exceptions being the initial vertex $v_0$ and the vertex at time $k$, when they may differ by 1. Since $\Lambda$ is a bipartite graph, we can divide $V(\Lambda)$ into two disjoint sets $U_1$ and $U_2$, such that every edge connects a vertex in $U_1$ to one in $U_2$; moreover, there are no edges connecting vertices in $U_1$, respectively in $U_2$. Then we can decouple $U_1$ and $U_2$ by
\begin{equation}
\label{bipartite}
\bigg|\sum_{v\in U_1} (X_k^v-\ell_0)-\sum_{v\in U_2} (X_k^v-\ell_0)\bigg|\le 2~~\mbox{and}~~[k/2]-1\le \sum_{v\in U_i} (X_k^v-\ell_0)\le [k/2]+1, i=1,2,
\end{equation}
where to get the second inequality in (\ref{bipartite}) we used the first inequality and $\sum_{v\in V(\Lambda)} (X_k^v-\ell_0)=k$. 

\vspace{0.5mm}

In view of (\ref{bipartite}), we consider next the possible values of the vector $R_k$. By abuse of notation, we will denote by ${\cal{L}}_{\nn,k}(\ell_k^3)$ the set of all possible values ${\vec{\ell}}_k=(\ell_k^1,\ldots,\ell_k^\nn)\in {\mathbb{N}}^{\nn}, 0\le\ell_k^\nn\le\ldots\le\ell_k^1,$ of $R_k$ when we \textit{fix} $R_k^3=\ell_k^3$. 
For $a,b,c, m\in\mathbb{N}$, we denote the possibly empty set
\begin{equation}
\label{bug}
{\cal{D}}_{m}^c(k;a;b):=\{{\vec{h}}=(h^1,\ldots, h^m)\in {\mathbb{N}}^m: 0\le h^m\le\ldots\le h^1\le c, [k/2]+a\le \sum_{i=1}^m h^i\le [k/2]+b\}.
\end{equation}
By (\ref{bipartite}), for all $(\ell_k^1,\ldots,\ell_k^\nn)\in {\cal{L}}_{\nn,k}(\ell_k^3)$ there exists $(\ell_k^{j_1},\ldots, \ell_k^{{j_{|U_1|}}})\in {\mathbb{N}}^{|U_1|}$, $0\le\ell_k^{j_{|U_1|}}\le\ldots\le\ell_k^{j_1}$, with
\begin{equation}
\label{hahaha}
[k/2]-1\le \sum_{j=1}^{|U_1|} \ell_k^{j_i}\le  [k/2]+1~~~\mbox{and}~~~[k/2]-1\le \sum_{j\in \{1,\ldots,\nn\}\setminus \{j_1,\ldots, j_{|U_1|}\}}^{\nn} \ell_k^{j_i}\le  [k/2]+1.
\end{equation}
Thus, we have $(\ell_k^{j_1},\ldots, \ell_k^{j_{|U_1|}})\in {\cal{D}}_{|U_1|}^\infty(k;-1; 1)$ and $({\vec{\ell}}_k^i)_{i\in I_{|U_2|}}\in  {\cal{D}}_{|U_2|}^\infty(k;-1; 1)$, where for any vectors ${\vec{h}}_1\in \mathbb{R}^{d_1},\ldots,{\vec{h}}_r\in\mathbb{R}^{d_r}, r\ge 1$, we denote by $({\vec{h}}_1,\ldots,{\vec{h}}_r)$ the vector formed by the combined components of all the vectors re-arranged in non-increasing order. 
From (\ref{hahaha}), we have
\begin{multline}
\label{hahaha2}
 {\cal{L}}_{\nn,k}(\ell_k^3)\subseteq \{(\vec{h},\vec{s}): \vec{h}\in {\cal{D}}_{|U_1|}^\infty(k;-1; 1), \vec{s}\in {\cal{D}}_{|U_2|}^\infty(k;-1; 1),\\
 ~\ell_k^3~\mbox{is a component either of}~\vec{h}~\mbox{or of}~\vec{s}\}=:{\cal{C}}_{|U_1|,|U_2|}(k,\ell_k^3).
\end{multline}

\vspace{1mm}

(a)  We will estimate the five terms in (\ref{eqmainverta}). The tricky term is the one from Step 1 in Proposition \ref{2orderstatv}, as the bounds obtained in Steps 2 and 3 therein are already tight enough, as we will see below. 

\textbf{Step 1:} We need to bound
\begin{equation}
\label{sum1vbpah}
\sum_{(\ell_k^1,\ell_k^2,\ell_k^3,\ldots, \ell_k^\nn)\in {\cal{L}}_{\nn,k}(\ell_k^3)}\frac{1}
{\prod_{i=3}^{\nn} w\left(\ell_k^i+\ell_0\right)}.
\end{equation}
In the above, we used our new definition of ${\cal{L}}_{\nn,k}(\ell_k^3)$, which includes $\ell_k^3$ in the summation, as this will not change the sum due to $\ell_k^3$ being fixed; however, in view of our notations above, this trick will allow us to take advantage of (\ref{bipartite}) and decouple the products over vertices in $U_1$ from those over $U_2$. To expand (\ref{sum1vbpah}) by means of (\ref{hahaha2}), we need to keep track of whether $\ell_k^1$ and $\ell_k^2$ (i.e. the two largest components in $(\vec{h},\vec{s})$) belong to vectors in ${\cal{D}}_{|U_1|}^\infty(k;-1; 1)$ or in ${\cal{D}}_{|U_2|}^\infty(k;-1; 1)$, as this will determine which components of $\vec{h}$ and $\vec{s}$ are missing in the products below.  Therefore, we have
\begin{multline}
\label{sum1vbpaha}
\sum_{(\ell_k^1,\ell_k^2,\ell_k^3,\ldots, \ell_k^\nn)\in {\cal{L}}_{\nn,k}(\ell_k^3)}\frac{1}
{\prod_{i=3}^{\nn} w\left(\ell_k^i+\ell_0\right)}\\
\le\sum_{(\vec{h},\vec{s})\in {\cal{C}}_{|U_1|,|U_2|}(k,\ell_k^3)\atop h^2\le s^1, s^2\le h^1}\frac{1}{\prod_{i=2}^{|U_1|} w\left(h^i+\ell_0\right)\prod_{j=2}^{|U_2|} w\left(s^j+\ell_0\right)}+\sum_{(\vec{h},\vec{s})\in {\cal{C}}_{|U_1|,|U_2|}(k,\ell_k^3)\atop s^1\le h^3}\frac{1}{\prod_{i=3}^{|U_1|} w\left(h^i+\ell_0\right)}\\
\times\frac{1}{\prod_{j=1}^{|U_2|} w\left(s^j+\ell_0\right)}
+\sum_{(\vec{h},\vec{s})\in {\cal{C}}_{|U_1|,|U_2|}(k,\ell_k^3)\atop h^1\le s^3}\frac{1}{\prod_{i=1}^{|U_1|} w\left(h^i+\ell_0\right)\prod_{j=3}^{|U_2|} w\left(s^j+\ell_0\right)}.
\end{multline}
The first term in (\ref{sum1vbpaha}) corresponds to the case when ${\cal{D}}_{|U_1|}^\infty(k;-1;1)$ and $ {\cal{D}}_{|U_2|}^\infty(k;-1;1)$ each  contains only one of the two largest components in $(\vec{h},\vec{s})$, the second term corresponds to the case when ${\cal{D}}_{|U_1|}^\infty(k;-1;1)$ contains both of the two largest components in $(\vec{h},\vec{s})$, and the third term corresponds to the case when ${\cal{D}}_{|U_2|}^\infty(k;-1;1)$ contains both of the two largest components in $(\vec{h},\vec{s})$.

To simplify computations, we will assume next that $\min(|U_1|, |U_2|)\ge 3$. If $\min(|U_1|,|U_2|)=1$, $\Lambda$ is a star graph, so $R_k^1=\sum_{i=2}^\nn R_k^i$ and (\ref{sum1vbpaha}) reduces to the first sum. If $\min(|U_1|, |U_2|)= 2$, either the second sum in (\ref{sum1vbpaha}) contains only the product over $|U_2|$ terms or the third sum contains only the product over $|U_1|$ terms, which can be dealt with similarly to the more general situation below in (\ref{bugger2}).

By definition, for any $(\vec{h},\vec{s})\in {\cal{C}}_{|U_1|,|U_2|}(k,\ell_k^3)$, the coordinate $\ell_k^3$ (i.e., the third largest component in $(\vec{h},\vec{s})$) is a component either in $\vec{h}\in {\cal{D}}_{|U_1|}^\infty(k;-1;1)$, or in $\vec{s}\in {\cal{D}}_{|U_2|}^\infty(k;-1;1)$. For the first term in (\ref{sum1vbpaha}), this reduces to checking whether $h^2=\ell_k^3$ or $s^2=\ell_k^3$. Thus, the first term in (\ref{sum1vbpaha}) becomes
\begin{eqnarray}
\label{bugger}
\lefteqn{\sum_{(\vec{h},\vec{s})\in {\cal{C}}_{|U_1|,|U_2|}(k,\ell_k^3)\atop h^2\le s^1,s^2\le h^1}\frac{1}{\prod_{i=2}^{|U_1|} w\left(h^i+\ell_0\right)\prod_{j=2}^{|U_2|} w\left(s^j+\ell_0\right)}}\nonumber\\
&=&\!\!\!\!\sum_{(\vec{h},\vec{s})\in {\cal{C}}_{|U_1|,|U_2|}(k,\ell_k^3)\atop h^2\le s^1,s^2\le h^1, h^2=\ell_k^3}\!\frac{1}{\prod_{i=2}^{|U_1|} w\left(h^i+\ell_0\right)\!\prod_{j=2}^{|U_2|} w\left(s^j+\ell_0\right)}+\!\!\!\!\sum_{(\vec{h},\vec{s})\in {\cal{C}}_{|U_1|,|U_2|}(k,\ell_k^3)\atop h^2\le s^1,s^2\le h^1, s^2=\ell_k^3}\!\frac{1}{\prod_{i=2}^{|U_1|} w\left(h^i+\ell_0\right)\!\prod_{j=2}^{|U_2|} w\left(s^j+\ell_0\right)}\nonumber\\
&\le&\frac{1}{w(\ell_k^3+\ell_0)}\bigg(\sum_{\vec{z}\in {\cal{D}}_{|U_1|-1}^{\infty}(k;-1-\ell_k^3; 1-\ell_k^3)\atop \vec{s}\in  {\cal{D}}_{|U_2|}^\infty(k;-1; 1)}\frac{1}{\prod_{i=2}^{|U_1|-1} w\left(z^i+\ell_0\right)\prod_{j=2}^{|U_2|} w\left(s^j+\ell_0\right)}\nonumber\\
&&\,\,\,\,\,\,\,\,\,\,\,\,\,\,\,\,\,\,\,+\sum_{\vec{h}\in {\cal{D}}_{|U_1|}^{\infty}(k;-1; 1)\atop \vec{y}\in  {\cal{D}}_{|U_2|-1}^\infty(k;-1-\ell_k^3; 1-\ell_k^3)}\frac{1}{\prod_{i=2}^{|U_1|} w\left(h^i+\ell_0\right)\prod_{j=2}^{|U_2|-1} w\left(y^j+\ell_0\right)}\bigg),
\end{eqnarray}
where for the inequality we removed the restriction that $h^2\le s^1, s^2\le h^1$; this will allow us in (\ref{bugger1}) to decouple the products with terms in ${\cal{D}}_{|U_1|-1}^{\infty}$ from those with terms in ${\cal{D}}_{|U_2|}^{\infty}$. We also used that for $\vec{h}\in {\cal{D}}_{|U_1|}^\infty(k;-1;1)$ with $h^2=\ell_k^3,$ the vector $\vec{z}:=(h^1,h^3,\ldots, h^{|U_1|-1}) \in {\cal{D}}_{|U_1|-1}^{\infty}(k;-1-\ell_k^3; 1-\ell_k^3)$; similarly,  for $\vec{s}\in {\cal{D}}_{|U_2|}^\infty(k;-1;1)$ with $s^2=\ell_k^3,$ the vector $\vec{y}:=(s^1,s^3,\ldots, s^{|U_2|-1}) \in {\cal{D}}_{|U_2|-1}^{\infty}(k;-1-\ell_k^3; 1-\ell_k^3)$. 
Expanding the first term in (\ref{bugger}) further, and in view of (\ref{bug}), we have for some $c(\ell_0,\nn)>0$
\begin{eqnarray}
\label{bugger1}
\lefteqn{\sum_{\vec{z}\in {\cal{D}}_{|U_1|-1}^{\infty}(k;-1-\ell_k^3; 1-\ell_k^3)\atop \vec{s}\in  {\cal{D}}_{|U_2|}^\infty(k;-1; 1)}\frac{1}{\prod_{i=2}^{|U_1|-1} w\left(z^i+\ell_0\right)\prod_{j=2}^{|U_2|} w\left(s^j+\ell_0\right)}}\nonumber\\
&=& \bigg(\sum_{\vec{z}\in {\cal{D}}_{|U_1|-1}^{\infty}(k;-1-\ell_k^3; 1-\ell_k^3)}\frac{1}{\prod_{i=2}^{|U_1|-1} w\left(z^i+\ell_0\right)}\bigg)\bigg(\sum_{\vec{s}\in  {\cal{D}}_{|U_2|}^\infty(k;-1; 1)}\frac{1}{\prod_{j=2}^{|U_2|} w\left(s^j+\ell_0\right)}\bigg)\nonumber\\
&\le&\bigg(\sum_{r=[k/2]-1-\ell_k^3}^{[k/2]+1-\ell_k^3}\sum_{(z^1,\ldots,z^{|U_1|-1}),\sum_{i} z^i=r\atop 0\le z^{|U_1|-1}\le\ldots\le z^1}\frac{1}{\prod_{i=2}^{|U_1|-1} w\left(z^i+\ell_0\right)}\bigg)\bigg(\sum_{r=[k/2]-1}^{[k/2]+1}\sum_{(s^1,\ldots,s^{|U_2|}),\sum_{j} s^j=r\atop 0\le s^{|U_2|}\le\ldots\le s^1}\frac{1}{\prod_{j=2}^{|U_2|} w\left(s^j+\ell_0\right)}\bigg)\nonumber\\
&\le& c(\ell_0,\nn),
\end{eqnarray}
where in each of the two products above the $w$ term coming from the largest component is missing (i.e., $w(z^1+\ell_0)$, respectively $w(s^1+\ell_0)$) and the summation is over all the terms (including the ones missing in the products), so the bound follows by the same arguments as in (\ref{yes1}). For the remaining term in (\ref{bugger}), we have by similar arguments
\begin{eqnarray}
\label{buggerbugger1}
\lefteqn{\sum_{\vec{h}\in {\cal{D}}_{|U_1|}^{\infty}(k;-1; 1)\atop \vec{y}\in  {\cal{D}}_{|U_2|-1}^\infty(k;-1-\ell_k^3; 1-\ell_k^3)}\frac{1}{\prod_{i=2}^{|U_1|} w\left(h^i+\ell_0\right)\prod_{j=2}^{|U_2|-1} w\left(y^j+\ell_0\right)}\bigg)} \nonumber\\
&\le&\bigg(\sum_{r=[k/2]-1}^{[k/2]+1}\sum_{(s^1,\ldots,s^{|U_1|}),\sum_{j} s^j=r\atop 0\le s^{|U_1|}\le\ldots\le s^1}\frac{1}{\prod_{j=2}^{|U_1|} w\left(s^j+\ell_0\right)}\bigg)\bigg(\sum_{r=[k/2]-1-\ell_k^3}^{[k/2]+1-\ell_k^3}\sum_{(z^1,\ldots,z^{|U_2|-1}),\sum_{i} z^i=r\atop 0\le z^{|U_2|-1}\le\ldots\le z^1}\frac{1}{\prod_{i=2}^{|U_2|-1} w\left(z^i+\ell_0\right)}\bigg)\nonumber\\
&\le& c'(\ell_0,\nn),
\end{eqnarray}
for some $c'(\ell_0,\nn)>0$. For the second term in (\ref{sum1vbpaha}), we have as in (\ref{bugger})
\begin{eqnarray}
\label{bugger2}
\lefteqn{\sum_{(\vec{h},\vec{s})\in {\cal{C}}_{|U_1|,|U_2|}(k,\ell_k^3)}\frac{1}{\prod_{i=3}^{|U_1|} w\left(h^i+\ell_0\right)\prod_{j=1}^{|U_2|} w\left(s^j+\ell_0\right)}}\nonumber\\
&\le&\frac{1}{w(\ell_k^3+\ell_0)}\bigg(\sum_{\vec{y}\in {\cal{D}}_{|U_2|-1}^{\infty}(k;-1-\ell_k^3; 1-\ell_k^3) \atop {\vec{h}\in  {\cal{D}}_{|U_1|}^{\ell_k^3}(k;-1; 1)}}\frac{1}{\prod_{i=3}^{|U_1|} w\left(h^i+\ell_0\right)\prod_{j=1}^{|U_2|-1} w\left(y^j+\ell_0\right)}\nonumber\\
&&\,\,\,\,\,\,\,\,\,\,\,\,\,\,\,\,\,\,\,+\sum_{\vec{z}\in {\cal{D}}_{|U_1|-1}^{\infty}(k;-1-\ell_k^3; 1-\ell_k^3)\atop {\vec{s}\in  {\cal{D}}_{|U_2|}^\infty(k;-1; 1)}}\frac{1}{\prod_{i=3}^{|U_1|-1} w\left(z^i+\ell_0\right)\prod_{j=1}^{|U_2|} w\left(s^j+\ell_0\right)}\bigg).\end{eqnarray}
We estimate next the first term in (\ref{bugger2}), the bounds for the second term following similarly. This requires a trick since there are two more terms in the summation over $h^i$ terms than in the corresponding product, so we cannot use  (\ref{yes1}). Since $\sum_{i=1}^{|U_2|-1} y^i\ge [k/2]-1-\ell_k^3\ge [k/2]-1-[k/3]$, we have $y^1\ge ([k/6]-1)/(|U_2|-1)\ge c(|U_2|) [k/3]  \ge c(|U_2|)h^1, c(|U_2|)>0$, with the last inequality due to $h^1\le\ell_k^3$. By decoupling now the terms in ${\cal{D}}_{|U_1|}^{\infty}$ from those in ${\cal{D}}_{|U_2|-1}^{\ell_k^3}$, we get 
\begin{eqnarray}
\label{sum1vpb2}
\lefteqn{\sum_{\vec{y}\in {\cal{D}}_{|U_2|-1}^{\infty}(k;-1-\ell_k^3; 1-\ell_k^3) \atop {\vec{h}\in  {\cal{D}}_{|U_1|}^{\ell_k^3}(k;-1; 1)}}\frac{1}{\prod_{j=1}^{|U_2|-1} w\left(y^j+\ell_0\right)}\times \frac{1}{\prod_{i=3}^{|U_1|} w\left(h^i+\ell_0\right)}}\nonumber\\
&\le&\bigg(\sum_{ \vec{y}\in {\cal{D}}_{|U_2|-1}^{\ell_k^3}(k;-1-\ell_k^3; 1-\ell_k^3) }\frac{(c(|U_2|))^{-1}y^1}{\prod_{j=1}^{|U_2|-1} w\left(y^j+\ell_0\right)}\bigg)\bigg(\sum_{r=[k/2]-1}^{[k/2]+1}\sum_{s=0}^r\sum_{(h^3,\ldots,h^{|U_1|}),\sum_{i=3}^{|U_1|} h^i=r-s\atop 0\le h^{|U_1|}\le\ldots\le h^3}\frac{1}{\prod_{i=3}^{|U_1|} w\left(h^i+\ell_0\right)}\bigg)\nonumber\\
&\le&\bigg(\sum_{r=[k/2]-1-\ell_k^3}^{[k/2]+1-\ell_k^3}\sum_{(y^1,\ldots,y^{|U_2|-1}),\sum_{j=1}^{|U_2|-1} y^j=r\atop 0\le y^{|U_2|-1}\le\ldots\le y^1\le\ell_k^3}\frac{(c(|U_2|))^{-1}y^1}{\prod_{j=1}^{|U_2|-1} w\left(y^j+\ell_0\right)}\bigg)\bigg(\sum_{r=[k/2]-1}^{[k/2]+1}\sum_{s=0}^rQ_{|U_1|-3}(r-s;\ell_0;\infty)\bigg)\nonumber\\
&\le& c(\ell_0, |U_2|)\bigg(\sum_{r=[k/2]-1-\ell_k^3}^{[k/2]+1-\ell_k^3}\,\,\,\,\sum_{y^1=0}^{\min(r,\ell_k^3)}\frac{y^1}{w(y^1+\ell_0)} Q_{|U_2|-2}(r-y^1;\ell_0;\infty)\bigg),
\end{eqnarray}
for some $c(\ell_0,|U_2|)>0$. To get the first inequality in (\ref{sum1vpb2}), we used in the second product that for each fixed $h^1+h^2=s$, there are less than $h^1\le(c(|U_2|))^{-1}y^1$ pairs $(h^1,h^2)$ to be summed over. For the second inequality, we expanded the first product in view of (\ref{bug}), and applied (\ref{eqndefq}) to the second product. For the third inequality, we used (\ref{eqndefq}) and that $y^1\le\min(r,\ell_k^3)$ in the inner sum of the first product to get a formula with respect to $y_1$ and $Q_{|U_2|-2}(r-y^1;\ell_0;\infty)$, and we utilized (\ref{rec2}) in the second product.

Similar estimates as in (\ref{bugger2}) and (\ref{sum1vpb2}) hold for the last term in (\ref{sum1vbpaha}). Therefore, by collecting all the estimates above, we have for some $c(\ell_0,|U_1|,|U_2|)>0$
\begin{eqnarray}
\label{sum1vbp1}
\lefteqn{\sum_{(\ell_k^1,\ell_k^2,\ell_k^3,\ldots, \ell_k^\nn)\in {\cal{L}}_{\nn,k}(\ell_k^3)}\frac{1}
{\prod_{i=3}^{\nn} w\left(\ell_k^i+\ell_0\right)}}\nonumber\\
&\le&\frac{c(\ell_0,|U_1|,|U_2|)}{w(\ell_k^3+\ell_0)}\bigg[\sum_{r=[k/2]-1-\ell_k^3}^{[k/2]+1-\ell_k^3}\sum_{y^1=0}^{\min(r,\ell_k^3)}\frac{y^1}{w(y^1+\ell_0)} \bigg(Q_{|U_2|-2}(r-y^1;\ell_0;\infty)+Q_{|U_1|-2}(r-y^1;\ell_0;\infty)\bigg)\nonumber\\
&&+\sum_{r=[k/2]-1}^{[k/2]+1}\sum_{y^1=0}^{\min(r,\ell_k^3)}\frac{y^1}{w(y^1+\ell_0)} \bigg(Q_{|U_2|-1}(r-y^1;\ell_0;\infty)+Q_{|U_1|-1}(r-y^1;\ell_0;\infty)\bigg)\bigg].
\end{eqnarray}
Up to this point, the proof has been common to both (a) and (b) as we have not used any of the specific assumptions on $w$ in each of the two cases. We note now that
\begin{eqnarray}
\label{sum1vb1}
\lefteqn{\frac{1}{w(\ell_k^3+\ell_0)}\sum_{y^1=0}^{\min(r,\ell_k^3)}\frac{y^1}{w(y^1+\ell_0)}Q_{|U_2|-2}(r-y^1;\ell_0;\infty)}\nonumber\\
&\le&\frac{(c(\ell_0))^{|U_2|-2}}{w(\ell_k^3+\ell_0)}\sum_{y^1=0}^{\ell_k^3}\frac{y^1}{w(y^1+\ell_0)}\le\frac{(c(\ell_0))^{|U_2|-2}(\ell_k^3)^{1/2}}{w(\ell_k^3+\ell_0)}\sum_{y^1=0}^{\ell_k^3}\frac{(y^1)^{1/2}}{w(y^1+\ell_0)}\le c'(\ell_0)\frac{(\ell_k^3)^{1/2}}{w(\ell_k^3+\ell_0)},
\end{eqnarray}
where $c'(\ell_0)>0$, for the first inequality we used (\ref{rec2}), and for the last inequality we used (\ref{supinfv1'}). We can reason likewise for the other terms in (\ref{sum1vbp1}) to get similar bounds.

\textbf{Step 2:} We have from (\ref{2vbc}) that, for some $\bar{c_1}(\ell_0,\nn), \tilde{c_1}(\ell_0,\nn)>0$, 
\begin{multline}
\label{2vbcr}
\sum_{(\ell_k^1,\ell_k^2,\ell_k^3,\ldots, \ell_k^\nn)\in {\cal{L}}_{\nn,k}(\ell_k^3)}
\frac{1}{w(\ell_k^3+\ell_0)w(\ell_k^1+\ell_0)}\sum_{j=4}^{\nn}\frac{1}{\prod_{i=4, i\neq j}^{\nn} w\left(\ell_k^i+\ell_0\right)}\\
\le \bar{c}(\ell_0,\nn)\frac{\ell_k^3}{w(\ell_k^3+\ell_0)} \bigg(\sum_{\ell_k^1=\max([{k}/\nn],\ell_k^3)}^{{k}-2\ell_k^3}\frac{1}{w(\ell_k^1+\ell_0)}\bigg)\le \tilde{c_1}(\ell_0,\nn)\frac{\ell_k^3}{w(\ell_k^3+\ell_0)}\frac{1}{k^{1/2}},
\end{multline}
as, in view of (\ref{supinfv1'}), we have 
$$\sum_{\ell_k^1=\max([{k}/\nn],\ell_k^3)}^{{k}-2\ell_k^3}\frac{1}{w(\ell_k^1+\ell_0)}\le \frac{1}{k^{1/2}}\sum_{\ell_k^1=\max([{k}/\nn],\ell_k^3)}^{{k}-2\ell_k^3}\frac{(\ell_k^1)^{1/2}}{w(\ell_k^1+\ell_0)}<\tilde{c_1}(\ell_0,\nn)\frac{1}{k^{1/2}}.$$
A similar reasoning gives for the term from (\ref{3vbc}) an upper bound of $\frac{\tilde{c_1}'(\ell_0,\nn)}{{{k}}^{1/2}}\sum_{i=\ell_k^3}^{{k}-\ell_k^3}\frac{1}{w(i+\ell_0)},$ where $\tilde{c_1}'(\ell_0,\nn)>0$. Thus, from Steps 1 and 2 we have
\begin{equation*}
{\mathbb{P}}^{\GG}(R_k^{3}=\ell_k^{3})\le C(w,\nn, \ell_0)\left[\frac{(\ell_k^3)^{1/2}}{w(\ell_k^3+\ell_0)}+ \frac{1}{{{k}}^{1/2}}\sum_{i=\ell_k^3}^{{k}-\ell_k^3}\frac{1}{w(i+\ell_0)}\right],
\end{equation*}
for some $C(w,\nn,\ell_0)>0$ which depends only on $w,\nn$ and $\ell_0$. Given the above bounds, the proof follows now the same arguments as the proof of Theorem \ref{mainvertfin} and will be omitted.

\vspace{0.5mm} 

(b)  Under assumption (\ref{bipbip}), we have in (\ref{sum1vbp1}) for each fixed $[k/2]-1\le r\le [k/2]+1$
\begin{eqnarray}
\lefteqn{\frac{1}{w(\ell_k^3+\ell_0)}\sum_{y^1=0}^{\min(r,\ell_k^3)}\frac{y^1}{w(y^1+\ell_0)}Q_{|U_2|-2}(r-y^1;\ell_0;\infty)}\nonumber\\
&\le&\frac{C(w)}{w(\ell_k^3+\ell_0)}\sum_{y^1=0}^{\min(r,\ell_k^3)}Q_{|U_2|-2}(r-y^1;\ell_0;\infty)\le\frac{C(\ell_0,w,\nn)}{w(\ell_k^3+\ell_0)},
\end{eqnarray}
for some $C(w), C(\ell_0,w,\nn)>0$, and where for the first inequality we used (\ref{bipbip}) and for second we used (\ref{rec2}). Arguing similarly for the other terms in (\ref{sum1vbp1}), we have for some $C_2(\ell_0,w,\nn)>0$
$$\sum_{(\ell_k^1,\ell_k^2,\ell_k^3,\ldots, \ell_k^\nn)\in {\cal{L}}_{\nn,k}(\ell_k^3)}\frac{1}
{\prod_{i=3}^{\nn} w\left(\ell_k^i+\ell_0\right)}\le \frac{C_2(\ell_0,w,\nn)}{w(\ell_k^3+\ell_0)}.$$
The above is enough to prove the statement for non-decreasing $w$ satisfying (\ref{bipbip}), since in this case the first term in (\ref{eqmainverta}) is an upper bound for each of the other four terms in (\ref{eqmainverta}); for more general $w$, we can reason as in Step 1 above, to bound the quantities in Steps 2 and 3 of Proposition \ref{2orderstatv}.
\end{proof}
We will next extend the result from Lemma \ref{specialcase} (a) to the more general case of triangle-free graphs.
 
 \begin{Lemma}
\label{triangle-free case}
Let $\Lambda$ be a finite triangle-free graph with $|\Lambda|=\nn\ge 4$. If $w$ satisfies
\begin{equation}
\label{supinfv1''}
\max_{v\in V(\GG)}\sum _{i\geq 1}^{\infty}\frac{i^{1/2}}{w(i+\ell_0^v)} < \infty,
\end{equation}
then the walk traverses exactly two random neighbouring attracting vertices at all large times a.s..
\end{Lemma}
\begin{proof}

\noindent  \textbf{Step 1:} For any three vertices $a,b,c\in V(\Lambda)$, if $a$ and $b$ are nearest neighbours, then $c$ can be nearest neighbour to at most one of $a$ and $b$ or else $a, b$ and $c$ would form a triangle.  Let $v_i, 1\le i\le \nn,$ be the $i$-th most visited vertex in $\Lambda$ by time $k$.

\noindent Fix $0\le R_k^3\le [k/3]$. There are now two cases to consider:

\noindent \textit{Case 1:} At least one of $v_1$ and $v_2$, say $v_1$, is not nearest neighbour to $v_3$.
As there are no visits between $v_1$ and $v_3$, we have $2 (X_k^{v_1}-\ell_0)+2R_k^3-1\le k$ and $k-R_k^3=\sum_{i\neq 3} R_k^i= (X_k^{v_1}-\ell_0)+(X_k^{v_2}-\ell_0)+\sum_{i=4}^\nn R_k^i\le  (X_k^{v_1}-\ell_0)+(X_k^{v_1}-\ell_0+\sum_{i=3}^\nn R_k^i)+\sum_{i=4}^\nn R_k^i.$ 
Thus, from the two inequalities combined we get
\begin{equation}
\label{case1aa}
[{k}/2]-R_k^3-(\nn-3)R_k^4\le X_k^{v_1}-\ell_0\le [({k}+1)/2]-R_k^3,
\end{equation}
which inequality we will use below if $[k/2]-R_k^3> (\nn-3)R_k^4$. Otherwise, we will use
\begin{equation}
\label{case1bb}
R_k^3\le X_k^{v_1}-\ell_0\le [({k}+1)/2]\le 1+R_k^3+(\nn-3)R_k^4.
\end{equation}
Putting (\ref{case1aa}) and (\ref{case1bb}) together gives that for each fixed $R_k^1+R_k^2$, the number of  $(R_k^1,R_k^2)$ terms is smaller or equal than $1+(\nn-3)R_k^4$.

\vspace{0.5mm}

\noindent\textit{Case 2:} Both $v_1$ and $v_2$ are nearest neighbours to $v_3$. Then there are no visits between $v_1$ and $v_2$, and we have $2 (X_k^{v_1}-\ell_0)+2(X_k^{v_2}-\ell_0)-1 \le k$,
from which it follows that
\begin{equation}
\label{case2aa}
2R_k^3\le (X_k^{v_1}-\ell_0)+(X_k^{v_2}-\ell_0)\le [({k}+1)/2].
\end{equation}
Moreover, we observe that since
$$k=(X_k^{v_1}-\ell_0)+(X_k^{v_2}-\ell_0)+\sum_{i=3}^\nn R_k^i\le  1+\sum_{i=3}^\nn R_k^i+\sum_{i=3}^\nn R_k^i\le 1+2R_k^3+2(\nn-3)R_k^4,$$ 
we have $[(k-1)/2]\le R_k^3+ (\nn-3)R_k^4.$ Combining this with (\ref{case2aa}), we obtain
$$R_k^3\le R_k^1+R_k^2\le 1+R_k^3+(\nn-3)R_k^4.$$
This gives again as in Case 1 that for each fixed $R_k^1+R_k^2$, the number of  $(R_k^1,R_k^2)$ terms is smaller or equal than $1+(\nn-3)R_k^4$. 

By applying the above bound in the first sum from (\ref{eqmainverta}), and with the same notation, we get 

\begin{eqnarray}
\label{sum1vb}
\lefteqn{\sum_{(\ell_k^1,\ell_k^2,\ell_k^4,\ldots, \ell_k^\nn)\in {\cal{L}}_{\nn,k}(\ell_k^3)}\frac{1}
{\prod_{i=3}^{\nn} w\left(\ell_k^i+\ell_0\right)}}\nonumber\\
&\le&\frac{\nn}{w(\ell_k^3+\ell_0)}\sum_{\ell_k^1+\ell_k^2=2\ell_k^3}^{{k}-\ell_k^3}\,\,\,\,\sum_{(\ell_k^4,\ldots, \ell_k^\nn)\in {\mathbb{N}}^{\nn-3}: \,0\le \ell_k^\nn\le\ldots\le\ell_k^3\atop \sum_{i=4}^{\nn}\ell_k^i=k-\ell_k^1-\ell_k^2-\ell_k^3}\frac{\ell_k^4}
{\prod_{i=4}^{\nn} w\left(\ell_k^i+\ell_0\right)}\nonumber\\
&\le& \frac{\nn}{w(\ell_k^3+\ell_0)}\sum_{s=2\ell_k^3}^{{k}-\ell_k^3}\,\,\,\,\sum_{\ell_k^4=0}^{\min(\ell_k^3, k-s-\ell_k^3)}\frac{\ell_k^4}{w(\ell_k^4+\ell_0)}\sum_{(\ell_k^5,\ldots, \ell_k^\nn)\in {\mathbb{N}}^{\nn-4}: \,0\le \ell_k^\nn\le\ldots\le\ell_k^4\atop \sum_{i=5}^{\nn}\ell_k^i=k-s-\ell_k^3-\ell_k^4}\frac{1}
{\prod_{i=5}^{\nn} w\left(\ell_k^i+\ell_0\right)}\nonumber\\
&\le&\frac{\nn}{w(\ell_k^3+\ell_0)}\sum_{s=2\ell_k^3}^{{k}-\ell_k^3}\,\,\,\sum_{\ell_k^4=0}^{\min(\ell_k^3, k-s-\ell_k^3)}\frac{\ell_k^4}{w(\ell_k^4+\ell_0)}\,\,Q_{\nn-4}(k-s-\ell_k^3-\ell_k^4;\ell_0;\infty).
\end{eqnarray}
For the first inequality in (\ref{sum1vb}), we used that for each fixed $\ell_k^1+\ell_k^2=s$, there are less than $\nn\ell_k^4$ terms $(\ell_k^1,\ell_k^2)$ to be summed over. For the second inequality, we split the inner sum from the first inequality into a double sum, first after $\ell_k^4$ and then after the remaining $(\ell_k^5,\ldots, \ell_k^\nn)$, given $\ell_k^4$; we also used here that $\sum_{i=5}^{\nn}\ell_k^i=k-s-\ell_k^3-\ell_k^4\ge 0$ and that $0\le\ell_k^4\le\ell_k^3$ to obtain $0\le\ell_k^4\le\min(\ell_k^3,k-s-\ell_k^3)$. For the last inequality we applied the definition of $Q_{\nn-4}$ from (\ref{eqndefq}).

By expanding (\ref{sum1vb}) further and recalling (\ref{minineq}), we get for $0\le\ell_k^3\le [k/4]$ (equivalent to $2\ell_k^3\le k-2\ell_k^3$)
\begin{eqnarray}
\label{sum1vb1}
\lefteqn{\sum_{(\ell_k^1,\ell_k^2,\ell_k^4,\ldots, \ell_k^\nn)\in {\cal{L}}_{\nn,k}(\ell_k^3)}\frac{1}
{\prod_{i=3}^{\nn} w\left(\ell_k^i+\ell_0\right)}}\nonumber\\
&\le&\frac{\nn}{w(\ell_k^3+\ell_0)}\bigg[\sum_{s=2\ell_k^3}^{{k}-2\ell_k^3}\,\,\sum_{\ell_k^4=0}^{\ell_k^3}\frac{\ell_k^4}{w(\ell_k^4+\ell_0)}\,\,Q_{\nn-4}(k-s-\ell_k^3-\ell_k^4;\ell_0;\infty)\nonumber\\
&&\,\,\,\,\,\,\,\,\,\,\,\,\,\,\,\,\,\,\,\,\,\,\,\,\,\,\,\,\,\,\,\,\,\,\,\,+\sum_{s=k-2\ell_k^3}^{k-\ell_k^3}\sum_{\ell_k^4=0}^{k-\ell_k^3-s}\frac{\ell_k^4}{w(\ell_k^4+\ell_0)}\,\,Q_{\nn-4}(k-s-\ell_k^3-\ell_k^4;\ell_0;\infty)\bigg]\nonumber\\
&\le&\frac{\nn}{w(\ell_k^3+\ell_0)}\!\!\sum_{\ell_k^4=0}^{\ell_k^3}\frac{\ell_k^4}{w(\ell_k^4+\ell_0)}\bigg[\!\sum_{s=2\ell_k^3}^{k-2\ell_k^3}Q_{\nn-4}(k-s-\ell_k^3-\ell_k^4;\ell_0;\infty)\!+\!\!\!\!\!\!\!\sum_{s=k-2\ell_k^3}^{{k}-\ell_k^3-\ell_k^4}Q_{\nn-4}(k-s-\ell_k^3-\ell_k^4;\ell_0;\infty)\bigg]\nonumber\\
&\le&\frac{2\nn c(\ell_0)^{\nn-4}}{w(\ell_k^3+\ell_0)}\sum_{\ell_k^4=0}^{\ell_k^3}\frac{\ell_k^4}{w(\ell_k^4+\ell_0)}\le\frac{2\nn c(\ell_0)^{\nn-4}(\ell_k^3)^{1/2}}{w(\ell_k^3+\ell_0)}\sum_{\ell_k^4=0}^{\ell_k^3}\frac{(\ell_k^4)^{1/2}}{w(\ell_k^4+\ell_0)}\le c'(\nn,\ell_0)\frac{(\ell_k^3)^{1/2}}{w(\ell_k^3+\ell_0)},
\end{eqnarray}
where we recall that $c(\ell_0)=\sum_{i=0}^\infty 1/w(i+\ell_0)$ and $c'(\nn,\ell_0)>0$. For the first inequality in (\ref{sum1vb1}), we used for the first inner sum that $\min(\ell_k^3, k-s-\ell_k^3)=\ell_k^3$ since $s\le k-2\ell_k^3$, and for the second inner sum $\min(\ell_k^3, k-s-\ell_k^3)=k-s-\ell_k^3$ since $s\ge k-2\ell_k^3$. For the second inequality we changed the summation order between $s$ and $\ell_k^4$, for the third we used (\ref{rec2}), and for the last inequality we used (\ref{supinfv1''}). We can argue similarly for $[k/4]\le\ell_k^3\le [k/3]$, when only the sum with $k-2\ell_k^3\le s\le k-\ell_k^3$ is possible, to obtain again an inequality as in (\ref{sum1vb1}).

The remaining four sums from (\ref{eqmainverta}) can be bounded similarly as in Step 2 from Lemma \ref{specialcase} (a).
\end{proof}


\subsection{Attraction sets on infinite connected graphs of bounded degree}
\label{InfGraphvert}

In this section we will prove Theorem \ref{maininfvert}. To begin with, we will state a result on the probability that the walk ever visits more than $n$ vertices. More precisely, we have

 \begin{Lemma}\label{Lem:Stuck_n_distance_vertex_a}
Let $\GG$ be an infinite connected graph of bounded degree. If $w$ satisfies (\ref{initialweightv}), then we have
\begin{equation}\label{Eq:Stuck_k_distance_vertex_a}
\lim_{n\rightarrow\infty}{\mathbb{P}}^{\GG}\bigg(\displaystyle \sup_{k \geq 1} \norm{I_k} > n\bigg)=0.
\end{equation} 
\end{Lemma}
\begin{proof}
The proof of this statement is just a simple application of Theorem 5.1 in \cite{BRS} for super-linear VRRW with initial weights $\ell_0^v\in\mathbb{N},v\in V(\GG)$, which theorem shows by means of a Rubin construction argument that with probability 1 the walk gets attracted to a finite graph. This result can be easily adapted to the more general case of initial weights $\ell_0^v\ge 0,v\in V(\GG),$ and weight functions satisfying (\ref{Azerov}) and (\ref{initialweightv}).
\end{proof}

\textbf{Proof of Theorem \ref{maininfvert}}

Given Theorem \ref{mainvertfin} and Lemma \ref{Lem:Stuck_n_distance_vertex_a}, the proof of Theorem \ref{maininfvert} follows by the same arguments as the proof of Theorem \ref{maininfedge} and will be omitted.

\qed

\textbf{Proof of Corollary \ref{zd}}

Given Lemmas \ref{specialcase} and \ref{Lem:Stuck_n_distance_vertex_a}, the proof of Corollary \ref{zd} follows by the same arguments as the proof of Theorem \ref{maininfedge} and will be omitted.

\qed

\textbf{Proof of Corollary \ref{tfg}}

Given Lemmas \ref{triangle-free case} and \ref{Lem:Stuck_n_distance_vertex_a}, the proof of Corollary \ref{tfg} follows by the same arguments as the proof of Theorem \ref{maininfedge} and will be omitted.

\qed

\section*{Appendix}

\begin{Lemma}
\label{gexistence}
Let $\left(p_l\right)_{l \geq 1}$ be a non-negative sequence of real numbers such that
\begin{equation}\label{Eq:Sum}
\displaystyle\sum_{l\geq 1}^{\infty}p_l<\infty. 
\end{equation}
Then there exists a function $g\colon \bN \rightarrow \left[0, \infty\right)$ such that 
\begin{enumerate}
\item [(i)] $g\left(\cdot\right)$ is increasing and $\textstyle \lim_{l \uparrow \infty}g\left(l\right)=\infty$.
\item [(ii)] $\textstyle M:=\sum _{l\geq 1}^{\infty}g\left(l\right)p_l < \infty$. 
\end{enumerate}
\end{Lemma}
\begin{proof} Since $\textstyle\sum_{l\geq 1}^{\infty}p_l<\infty,$ there exists an increasing sequence $\left(n_l\right)_{l \geq 1}$ such that 
\begin{equation}\label{Eq:Choiceseq}
\displaystyle\sum_{n \geq n_l}^{\infty}p_n <\frac{1}{2^l}.
\end{equation}
Choose $N\in \bN$ large enough. Let $N_m:=\left[n_{N^m}, n_{N^{\left(m+1\right)}}-1\right)\bigcap \bN$ for every $m \in \bN$. Define $g\colon \bN \rightarrow\left[0, \infty\right)$ by
\begin{equation}\label{Eq: def_g}
 g\left(l\right)=\begin{cases}
2^{m}, &\text{for }l \in N_m \\
1, & \text{for } l \leq n_N.
\end{cases} 
\end{equation}
It is easy to see that 
\begin{equation*}\label{Eq: Sum_g}
\displaystyle \sum_{l \in N_m}g\left(l\right)p_l =2^m \displaystyle \sum_{l \in N_m}p_l < \frac{1}{2^{N^{m}-m}}.
\end{equation*}
Therefore
\begin{equation*}
\displaystyle \sum_{l \geq 1}g\left(l\right)p_l< \infty. 
\end{equation*}
This completes the proof.
\end{proof}

The proof in the above Lemma can be easily extended to the following more general setting
\begin{Lemma}
\label{gexistencem}
Fix $n\ge 2$ and let $\left(p_l^i\right)_{l \geq 1}, 1\le i\le n,$ be non-negative sequences of real numbers such that for all $1\le i\le n$, we have
\begin{equation}\label{Eq:Sum}
\displaystyle\sum_{l\geq 1}^{\infty}p_l^i<\infty. 
\end{equation}
Then there exists a function $g\colon \bN \rightarrow \left[0, \infty\right)$ such that 
\begin{enumerate}
\item [(i)] $g\left(\cdot\right)$ is increasing and $\textstyle \lim_{l \uparrow \infty}g\left(l\right)=\infty$.
\item [(ii)] $\textstyle M:=\sup_{1\le i\le n}\sum _{l\geq 1}^{\infty}g\left(l\right)p^i_l < \infty$. 
\end{enumerate}
\end{Lemma}


\subsection*{Acknowledgments}

C.C. thanks David Brydges for numerous comments and suggestions which greatly improved the presentation of the manuscript, and Tyler Helmuth and Richard Pymar for fruitful discussions during her work on the paper. C.C. also thanks Olivier Raimond for suggesting that she can use the order statistics method to extend the result in Corollary \ref{zd} (a) to general triangle-free graphs.

We warmly thank an anonymous referee for a very careful reading of the paper and for an absolutely outstanding refereeing job.

\end{document}